\documentclass[a4paper,12pt]{article}

\usepackage[text={6.5in,8.4in},centering]{geometry}
\usepackage{lmodern}
\RequirePackage[OT1]{fontenc}
\RequirePackage{amsthm,amsmath,amssymb,amscd}
\RequirePackage[numbers]{natbib}
\RequirePackage[colorlinks,citecolor=blue,urlcolor=blue]{hyperref}
\usepackage{enumerate}
\usepackage{multirow}
\usepackage{datetime}
\usepackage{etex}
\makeatother
\numberwithin{equation}{section}
\allowdisplaybreaks

\theoremstyle{plain}
\newtheorem{theorem}{Theorem}[section]

\newtheorem{lemma}[theorem]{Lemma}
\newtheorem{proposition}[theorem]{Proposition}

\theoremstyle{definition}
\newtheorem{definition}[theorem]{Definition}

\newtheorem{remark}[theorem]{Remark}

\newtheorem{example}[theorem]{Example}

\newcommand{\E}{\mathbb{E}}
\newcommand{\W}{\dot{W}}
\newcommand{\ud}{\ensuremath{\mathrm{d}}}

\newcommand{\Indt}[1]{1_{\left\{#1 \right\}}}
\newcommand{\Norm}[1]{\left|\left|  #1   \right|\right|}

\newcommand{\Itos}{It\^{o}'s }

\newcommand{\InPrd}[1]{\left\langle #1 \right\rangle}

\newcommand{\calB}{\mathcal{B}}

\newcommand{\calF}{\mathcal{F}}
\newcommand{\calG}{\mathcal{G}}
\newcommand{\calK}{\mathcal{K}}
\newcommand{\calH}{\mathcal{H}}

\newcommand{\calM}{\mathcal{M}}
\newcommand{\calN}{\mathcal{N}}

\newcommand{\calP}{\mathcal{P}}

\newcommand{\R}{\mathbb{R}}

\newcommand{\Erf}{\ensuremath{\mathrm{erf}}}
\newcommand{\Erfc}{\ensuremath{\mathrm{erfc}}}

\DeclareMathOperator{\Lip}{\mathit{L}}
\DeclareMathOperator{\LIP}{Lip}

\DeclareMathOperator{\Vip}{\overline{\varsigma}}

\DeclareMathOperator{\vv}{\varsigma}

\title{{\bf H\"older-continuity for the nonlinear stochastic heat equation 
with rough initial conditions}}

\author{{\bf Le Chen$^*$} and {\bf Robert C. Dalang\footnote{Research
partially supported by the Swiss National Foundation for Scientific
Research.}}\\
\\
\it\small Institut de math\'ematiques\\
\it\small \'Ecole Polytechnique F\'ed\'erale de Lausanne\\
\it\small Station 8, \\
\it\small CH-1015 Lausanne, \\
\it\small Switzerland\\
\small \textit{e-mails:}
le.chen@epfl.ch, robert.dalang@epfl.ch\\
\small \textit{Phone:} +41 21 693 25 88\quad  
\textit{Fax:} +41 21 693 25 50
\small
\date{}
}

\begin{document}

\maketitle
\begin{center}
\begin{minipage}[rct]{5 in}
\footnotesize \textbf{Abstract:} 
We study space-time regularity of the solution of the nonlinear stochastic heat 
equation in one spatial dimension driven by space-time white noise, with a 
rough initial condition. This initial condition is a locally finite measure 
$\mu$ with, possibly, 
exponentially growing tails. We show how this regularity depends, in a 
neighborhood of $t=0$, on the regularity of the initial  condition.  On compact 
sets in which $t>0$, the 
classical H\"older-continuity exponents $\frac{1}{4}-$ in time and 
$\frac{1}{2}-$ in space remain valid. However, on compact sets that include 
$t=0$, the H\"older continuity of the solution is 
$\left(\frac{\alpha}{2}\wedge \frac{1}{4}\right)-$ in time and 
$\left(\alpha\wedge \frac{1}{2}\right)-$ in space, provided $\mu$ is absolutely 
continuous with an $\alpha$-H\"older continuous density.

\vspace{2ex}

\vspace{2ex}
\vspace{4ex}
\end{minipage}
\end{center}


\section{Introduction}

Over the last few years, there has been considerable interest in the stochastic heat equation with non-smooth initial data:
\begin{align}\label{EH:Heat}
\begin{cases}
\left(\frac{\partial }{\partial t} - \frac{\nu}{2}
\frac{\partial^2 }{\partial x^2}\right) u(t,x) =  \rho(u(t,x))
\:\dot{W}(t,x),&
x\in \R,\; t \in\R_+^*, \\
\quad u(0,\cdot) = \mu(\cdot)\;.
\end{cases}
\end{align}
In this equation, $\dot{W}$ is a space-time white noise, $\rho: \R \to \R$ is a globally Lipschitz function and
$\R_+^*=\;]0,\infty[\:$. The initial data $\mu$ is a signed Borel measure, 
which we assume belongs to the set
\[
   \calM_H(\R) := \left\{\text{signed Borel measures $\mu$, s.t.}\: 
\int_\R e^{-a x^2} |\mu|(\ud x)<+\infty,\: \text{for all $a>0$}\right\}.
\]
In this definition, $|\mu|:= \mu_+ + \mu_-$, where $\mu=\mu_+-\mu_-$ and 
$\mu_\pm$ are the two non-negative Borel measures with disjoint support that 
provide the Jordan decomposition of $\mu$.
The set $\calM_H(\R)$ can be equivalently characterized by the condition that
\begin{align}\label{EH:J0finite}
\left(|\mu| * G_\nu(t,\cdot)\right)
(x) = \int_\R G_\nu(t,x-y) |\mu|(\ud y)<+\infty\;, \quad \text{for all $t>0$ 
and $x\in\R$},
\end{align}
where $*$ denotes the convolution in the space variable and $G_\nu(t,x)$ is 
the one-dimensional heat kernel function
\begin{align}
G_\nu(t,x) := \frac{1}{\sqrt{2\pi \nu t}} \exp\left\{-\frac{x^2}{2\nu
t}\right\},\quad (t,x)\in\R_+^*\times\R\:.
\end{align}
Therefore, $\calM_H(\R)$ is precisely the set of initial conditions for which the homogeneous heat equation has a solution for all time.

   The use of non-smooth initial data is initially motivated by the {\it parabolic 
Anderson model} (in which $\rho(u) = u$) with initial condition given by the 
Dirac delta 
function $\mu = \delta_0$ (see \cite{BertiniCancrini94Intermittence}, and more 
recently, 
\cite{FoondunKhoshnevisan08Intermittence, ConusKhosh10Weak,ConusEct12Initial}). 
These papers are mainly 
concerned with the study of the intermittency property, which is a property 
that concerns the behavior of moments of the solution $u(t,x)$. Some very 
precise moment estimates have also been recently obtained by the authors in 
\cite{ChenDalang13Heat}.

   In this paper, we are interested in space-time regularity of the  sample paths $(t,x) \mapsto u(t,x)$, and, in particular, in how this regularity depends, in a neighborhood of $\{0\}\times \R$, on the regularity of the initial condition $\mu$.

Given a subset $D\subseteq \R_+\times\R$ and positive constants 
$\beta_1,\beta_2$, denote by $C_{\beta_1,\beta_2}(D)$ the set of functions $v: 
\R_+ \times \R \to \R$ with the following property: for each compact subset 
$\tilde D \subset D$, there is a finite constant $c$ such that for all $(t,x)$ 
and $(s,y)$ in $\tilde D$,
\[
   \vert v(t,x) - v(s,y) \vert \leq c \left(\vert t-s\vert^{\beta_1} + \vert x-y \vert^{\beta_2}\right).
\]
Let
\[
   C_{\beta_1-,\beta_2-}(D) := \cap_{\alpha_1\in \;\left]0,\beta_1\right[} \cap_{\alpha_2\in \;\left]0,\beta_2\right[} C_{\alpha_1,\alpha_2}(D)\;.
\]
When the measure $\mu$ has a bounded density $f$ with respect to Lebesgue 
measure, then the initial condition is written $u(0,x) = f(x)$, for all $x \in 
\R$. When $f$ is bounded, then the H\"older continuity of $u$ was already 
studied in \cite[Corollary 3.4, p.318]{Walsh86}. In 
\cite{BertiniCancrini94Intermittence}, it is stated, for the parabolic Anderson 
model, that if the initial data satisfies
\[
\sup_{t\in [0,T]} \sup_{x\in\R} \sqrt{t}
\left(\mu*G_\nu(t,\circ)\right)(x)<\infty,\quad\text{for all $T>0$},
\]
then $u\in C_{\frac{1}{4}-,\frac{1}{2}-}(\R_+\times\R)$, a.s.
In \cite{PospisilTribe07Parameter,Shiga94Two}, this result is extended to the 
case where the initial data is a continuous function with tails that grow at 
most exponentially at $\pm \infty$.

   Sanz-Sol\'e and Sarr\`a \cite{SanSoleSarra99Holder} considered the stochastic heat equation over $\R^d$ with spatially homogeneous colored noise which is white in time. Assuming that the spectral measure $\tilde{\mu}$ of the noise satisfies
\begin{align}\label{EH:Holder-mu}
\int_{\R^d}\frac{\tilde{\mu}(\ud
\xi)}{\left(1+|\xi|^2\right)^\eta}<+\infty,
\quad\text{for some $\eta\in \;]0,1[$,}
\end{align}
they proved that if the initial data is a bounded $\rho$-H\"older continuous function for some $\rho\in \;]0,1[$, then 
\[
u\in C_{\frac{1}{2}(\rho\wedge (1-\eta))-, \rho\wedge
(1-\eta)-}
\left(\R_+\times\R\right)\;,\quad\text{a.s.}\;,
\]
where $a\wedge b := \min(a,b)$.
For the case of space-time white noise on $\R_+\times\R$, the spectral measure $\tilde{\mu}$ is Lebesgue measure and hence the exponent $\eta$ in \eqref{EH:Holder-mu} (with $d=1$) can take the value $\frac{1}{2}-\epsilon$ for any $\epsilon>0$. Their result (\cite[Theorem 4.3]{SanSoleSarra99Path}) implies that
\[
u\in C_{\left(\frac{1}{4}\wedge\frac{\rho}{2}\right)-,
\left(\frac{1}{2}\wedge\rho\right)-}\left(\R_+\times\R\right)\;,\quad\text{
a.s.}
\]

   More recently, in the paper \cite[Lemma 9.3]{ConusEct12Initial}, assuming that the initial condition $\mu$ is a finite measure, Conus {\it et al} obtain tight upper bounds on moments of $u$ and bounds on moments of spatial increments of $u$ at fixed positive times: in particular, they show that $u$ is H\"older continuous in $x$ with exponent $\frac{1}{2} - \epsilon$.
	
	Finally, in the papers \cite{DalangKhNualart07HittingAdditive, DalangKhNualart09HittingMult}, Dalang, Khoshnevisan and Nualart considered a system of stochastic heat equations with vanishing initial conditions driven by space-time white noise, and proved that $u \in C_{\frac{1}{4}-,\frac{1}{2}-}\left(\R_+\times\R\right)$.

   The purpose of this paper is to extend the above results to the case where 
 $\mu \in \calM_H(\R)$. In particular, we show that $u \in 
C_{\frac{1}{4}-,\frac{1}{2}-}\left(\R_+^*\times\R\right)$. Indeed, it is 
necessary to exclude the line $\{0\}\times \R$ unless the initial data $\mu$ has 
a density $f$ that is sufficiently smooth (see part (2) of Theorem 
\ref{T2:Holder}). Indeed, in this case, the regularity of $u$ in the neighborhood of 
$t=0$ can be no better than the regularity of $f$. 
	
	Recall that the rigorous interpretation of \eqref{EH:Heat}, used in \cite{ChenDalang13Heat}, is the following integral equation:
\begin{equation}\label{EH:WalshSI}
\begin{aligned}
u(t,x) &= J_0(t,x) + I(t,x),\cr
I(t,x) &= \iint_{[0,t]\times\R} G_\nu\left(t-s,x-y\right)
\rho\left( u\left(s,y\right) \right)
W\left(\ud s,\ud y\right),
\end{aligned}
\end{equation}
where $J_0(t,x) := \left(\mu * G_\nu(t,\cdot)\right)(x)$, and the stochastic integral is interpreted in the sense of Walsh \cite{Walsh86}. The regularity of $(t,x) \mapsto J_0(t,x)$ is classical (see Lemma \ref{LH:J0Cont}), so the main effort is to understand the H\"older-regularity of $(t,x) \mapsto I(t,x)$ at $t=0$. In general, this function is not H\"older continuous at $t=0$, but it is H\"older continuous away from the origin (see Theorem \ref{T2:Holder} and Proposition \ref{rd:moment}). We also show that for certain locally unbounded choices of the density $f$, the optimal H\"older exponent of $t \mapsto I(t,x)$ (for $t \in [0,1]$, say) may be arbitrarily close to $0$ (see Proposition \ref{P:Holder-Ubd}).

   The difficulties for proving the H\"older continuity of $u$ lie in part in the fact that
for initial data satisfying \eqref{EH:J0finite}, $\E\left[|u(t,x)|^p\right]$ 
need not be bounded over $x\in\R$, and mainly in the fact that the initial data is 
irregular. Indeed, standard techniques, which isolate the effects of initial 
data by using the $L^p(\Omega)$-boundedness of the solution, fail in our case 
(see Remark \ref{RH:Standard}). Instead, the initial data play an active role 
in our proof. We also note that Fourier transform techniques are not directly 
applicable here because $\mu$ need not be a tempered measure.

  Finally, it is natural to ask in what sense the measure $\mu$ is indeed the initial condition for the stochastic heat equation? We show in Proposition \ref{PH:Weak} that $u(t,\cdot)$ converges weakly (in the sense of distribution theory) to $\mu$ as $t \downarrow 0$, so that $\mu$ is the initial condition of \eqref{EH:Heat} in the classical sense used for deterministic p.d.e.'s \cite[Chapter 7, Section 1]{friedman}.

The paper is structured as follows. In Section \ref{SH:Preliminaries}, we recall 
the results of \cite{ChenDalang13Heat} that we need here. Our main results are 
stated in Section \ref{SH:MainResult}. The proofs are presented 
in Section \ref{SH:Proofs}. Finally, some technical lemmas are listed in the appendix.

\section{Some preliminaries}\label{SH:Preliminaries}

Let $W=\left\{
W_t(A),\:A\in\calB_f\left(\R\right),t\ge 0 \right\}$
be a space-time white noise
defined on a complete probability space $(\Omega,\calF,P)$, where
$\calB_f\left(\R\right)$ is the
collection of Borel sets with finite Lebesgue measure.
Let  
\[
\calF_t^0 = \sigma\left(W_s(A),\:0\le s\le
t,\:A\in\calB_f\left(\R\right)\right)\vee
\calN,\quad t\ge 0,
\]
be the natural filtration of $W$ augmented by the $\sigma$-field $\calN$ 
generated
by all $P$-null sets in $\calF$.
For $t\ge 0$, define $\calF_t := \calF_{t+}^0 = \wedge_{s>t}\calF_s^0$.
In the following, we fix the filtered
probability space $\left\{\Omega,\calF,\{\calF_t,\:t\ge0\},P\right\}$.
We use $\Norm{\cdot}_p$ to denote the
$L^p(\Omega)$-norm ($p\ge 1$).
With this setup, $W$ becomes a worthy martingale measure in the sense of Walsh
\cite{Walsh86}, and $\iint_{[0,t]\times\R}X(s,y) W(\ud s,\ud y)$ is
well-defined in this reference for a suitable class of random fields
$\left\{X(s,y),\; (s,y)\in\R_+\times\R\right\}$.

In this paper, we use $\star$ to denote the simultaneous convolution in both 
space and time variables.
\begin{definition}\label{D2:Solution}
A process $u=\left(u(t,x),\:(t,x)\in\R_+^*\times\R\right)$  is called a 
{\it random field solution} to \eqref{EH:WalshSI} if
the following four conditions are satisfied:
\begin{enumerate}[(1)]
 \item $u$ is adapted, i.e., for all $(t,x)\in\R_+^*\times\R$, $u(t,x)$ is
$\calF_t$-measurable;
\item $u$ is jointly measurable with respect to
$\calB(\R_+^*\times\R)\times\calF$;
\item $\left(G_\nu^2 \star \Norm{\rho(u)}_2^2\right)(t,x)<+\infty$
for all $(t,x)\in\R_+^*\times\R$, and
the function $(t,x)\mapsto I(t,x)$ mapping $\R_+^*\times\R$ into
$L^2(\Omega)$ is continuous;
\item $u$ satisfies \eqref{EH:WalshSI} a.s.,
for all $(t,x)\in\R_+^*\times\R$.
\end{enumerate}
\end{definition}

Assume that $\rho:\R\mapsto \R$ is globally Lipschitz
continuous with Lipschitz constant $\LIP_\rho>0$.
We consider the following growth conditions on $\rho$:
for some constants $\Lip_\rho>0$ and $\Vip \ge 0$,
\begin{align}\label{EH:LinGrow}
|\rho(x)|^2 \le \Lip_\rho^2 \left(\Vip^2 +x^2\right),\qquad \text{for
all $x\in\R$}\;.
\end{align}
Note that $\Lip_\rho\le \sqrt{2}\LIP_\rho$, and the inequality may be strict.
Of particular importance is the linear case (the {\it parabolic
Anderson model}): $\rho(u)=\lambda u$ with $\lambda\ne 0$, which is a
special case of the following quasi-linear growth condition:
for some constant $\vv\ge 0$,
\begin{align}\label{EH:qlinear}
|\rho(x)|^2= \lambda^2\left(\vv^2+x^2\right),\qquad  \text{for
all $x\in\R$}\;.
\end{align}

Define the kernel functions:
\begin{align}
\label{E:K}
 \calK(t,x)=\calK(t,x;\nu,\lambda)&:=G_{\frac{\nu}{2}}(t,x)
\left(\frac{\lambda^2}{\sqrt{4\pi\nu t}}+\frac{\lambda^4}{2\nu}
\: e^{\frac{\lambda^4 t}{4\nu}}\Phi\left(\lambda^2
\sqrt{\frac{t}{2\nu}}\right)\right),\\
\label{E2:H}
\calH(t)=\calH(t;\nu,\lambda)&:=\left(1\star \calK \right)(t,x) = 2 
e^{\frac{\lambda^4\:
t}{4\nu}}\Phi\left(\lambda^2\sqrt{\frac{t}{2\nu}}\right)-1,
\end{align}
where $\Phi(x)=\int_{-\infty}^x (2\pi)^{-1/2}e^{-y^2/2}\ud y$, and 
the formula on the right-hand side is explained in 
\cite[(2.18)]{ChenDalang13Heat}.
Some functions related to $\Phi(x)$ are the error functions $\Erf(x)= 
\frac{2}{\sqrt{\pi}}\int_0^x e^{-y^2}\ud y$ and $\Erfc(x)=1-\Erf(x)$. Clearly, 
$\Phi(x)= \left(1+\Erf(x/\sqrt{2})\right)/2$. 

Let $z_p$ be the universal constant in the Burkholder-Davis-Gundy inequality 
(see \cite[Theorem 1.4]{ConusKhosh10Farthest}, in particular, $z_2=1$) which 
satisfies $z_p\le 2\sqrt{p}$ for all $p\ge 2$. Let $a_{p,\Vip}$ be the constant 
defined by
\begin{align*}
a_{p,\Vip} \: :=\:
\begin{cases}
2^{(p-1)/p}& \text{if $\Vip\ne 0,\: p>2$},\cr
\sqrt{2} & \text{if $\Vip =0,\: p>2$},\cr
1 & \text{if $p=2$}.
\end{cases}
\end{align*}
Notice that $a_{p,\Vip}\in [1,2]$.
Denote $\overline{\calK}(t,x):= \calK(t,x;\nu,\Lip_\rho)$,
$\widehat{\calK}_p(t,x)= \calK(t,x; \nu, a_{p,\Vip} z_p \Lip_\rho)$ and 
$\overline{\calH}(t):= \calH(t;\nu,\Lip_\rho)$,
$\widehat{\calH}_p(t)= \calH(t; \nu, a_{p,\Vip} z_p \Lip_\rho)$.

The following theorem is mostly taken from \cite[Theorem 
2.4]{ChenDalang13Heat}, 
except that \eqref{E2:SecMom-Delta} comes from \cite[Corollary 
2.8]{ChenDalang13Heat}.

\begin{theorem}[Existence, uniqueness and moments]
\label{T2:ExUni}
Suppose that
the function $\rho$ is Lipschitz continuous and satisfies \eqref{EH:LinGrow}, 
and $\mu\in\calM_H(\R)$.
Then the stochastic integral equation \eqref{EH:WalshSI} has a random
field solution
$u=\{u(t,x),\: (t,x)\in\R_+^*\times\R\}$. Moreover:\\
(1) $u$ is unique (in the sense of versions).\\
(2) $(t,x)\mapsto u(t,x)$ is $L^p(\Omega)$-continuous for all integers $p\ge
2$.\\
(3) For all even integers $p\ge 2$, all $t>0$ and $x,y\in\R$,
\begin{align}
\label{E2:SecMom-Up}
\Norm{u(t,x)}_p^2 \le
\begin{cases}
J_0^2(t,x) + \left(J_0^2\star \overline{\calK} \right) (t,x) +
\Vip^2  \overline{\calH}(t),& \text{if $p=2$,}\cr
2J_0^2(t,x) + \left(2J_0^2\star \widehat{\calK}_p \right) (t,x) +
\Vip^2  \widehat{\calH}_p(t),& \text{if $p>2$.}
\end{cases}
\end{align}
(4) In particular, if $|\rho(u)|^2=\lambda^2
\left(\vv^2+u^2\right)$, then for all $t>0$ and $x,y\in\R$,
\begin{align}
 \label{E2:SecMom}
 \Norm{u(t,x)}_2^2 = J_0^2(t,x) + \left(J_0^2 \star \calK \right) (t,x)+ \vv^2
\: \calH(t).
\end{align}
Moreover, if $\mu= \delta_0$ (the Dirac delta function), then
\begin{align}
 \label{E2:SecMom-Delta}
\Norm{u(t,x)}_2^2 = \frac{1}{\lambda^2} \calK(t,x) + \vv^2 \calH(t).
\end{align}
\end{theorem}

The next lemma is classical. A proof can be found in \cite[Lemma 2.20]{ChenDalang13Heat}.

\begin{lemma}\label{LH:J0Cont}
The function $(t,x)\mapsto J_0(t,x)=\left(\mu * G_\nu(t,\cdot)\right)(x)$
 with  $\mu\in\calM_H(\R)$ is smooth for $t>0$:  $J_0(t,x) \in
C^{\infty}\left(\R_+^*\times\R\right)$.
If, in addition,  $\mu(\ud x) = f(x)\ud
x$
where $f$ is continuous, then $J_0$ is continuous up to $t=0$:
$J_0 \in
C^{\infty}\left(\R_+^*\times\R\right)
 \;\cap\;
C\left(\R_+\times\R\right)$, and if $f$ is
$\alpha$-H\"older continuous, then
$J_0 \in
C^{\infty}\left(\R_+^*\times\R\right)
 \;\cap\;
C_{\alpha/2,\alpha}\left(\R_+\times\R\right)$.
\end{lemma}

For $p\ge 2$ and $X\in L^2\left(\R_+\times\R, L^p(\Omega)\right)$, set
\begin{align*}
 \Norm{X}_{M, p}^2 & :
= \iint_{\R_+^*\times\R}\Norm{X\left(s,y\right)}_p^2\ud s\ud y<+\infty\;.
\end{align*}
When $p=2$, we write $\Norm{X}_M$ instead of $\Norm{X}_{M,2}$.
In \cite{Walsh86}, $\iint X\ud W$ is defined for {\em predictable} $X$
such that $\Norm{X}_M<+\infty$. Let 
$\calP_p$
denote the closure in $L^2\left(\R_+\times\R, L^p(\Omega)\right)$ of simple
processes. 
Clearly, $\calP_2\supseteq \calP_p \supseteq \calP_q$ for $2\le p\le q<+\infty$,
and according to \Itos isometry, $\iint X\ud W$ is well-defined for all
elements of $\calP_2$. The next lemma, taken from \cite[Lemma 
2.14]{ChenDalang13Heat}, gives easily verifiable conditions
for checking that $X\in\calP_2$. 
In the following, we will use $\cdot$ and
$\circ$ to denote the time and space dummy variables respectively.
\begin{lemma}
\label{LH:Lp}
Let $\calG(s,y)$ be a deterministic measurable function from $\R_+^* \times \R$ 
to $\R$ and let  $Z=\left(Z\left(s,y\right),\:
\left(s,y\right)\in \; \R_+^*\: \times \R\right)$ be a process with the 
following properties:
\begin{enumerate}[(1)]
 \item $Z$ is adapted and jointly measurable with respect to
$\calB(\R^2)\times\calF$;
\item $\E\left[\iint_{[0,t]\times\R} \calG^2\left(t-s,x-y\right)
Z^2\left(s,y\right)\ud
s\ud
y\right]<\infty$,
for all $(t,x)\in\R_+\times\R$.
\end{enumerate}
Then for each $(t,x)\in\R_+\times\R$, 
the random field $\left(s,y\right)\in \,]0,t[\: \times\R \mapsto 
\calG\left(t-s,x-y\right) Z\left(s,y\right)$ belongs to $\calP_2$ and so 
the stochastic convolution
\begin{align*}
(\calG\star Z\dot{W})(t,x) := \iint_{[0,t]\times\R}
\calG\left(t-s,x-y\right)Z\left(s,y\right)W\left(\ud s,\ud y\right)
\end{align*}
is a well-defined Walsh integral and the random field $\calG\star Z\W$ is 
adapted.
Moreover, for all even integers $p\ge 2$ and $(t,x)\in\R_+ \times\R$,
\[
\Norm{(\calG\star Z\dot{W})(t,x)}_p^2
\le z_{p}^{2} \Norm{\calG(t-\cdot,x-\circ) Z(\cdot,\circ)}_{M,p}^2.
\]
\end{lemma}

\section{Main results} \label{SH:MainResult}

If the initial data is of the form $\mu(\ud x)=f(x)\ud x$, where $f$ is a
bounded function, then it is well-known (see \cite{Walsh86}) that the solution 
$u$ is  bounded in
$L^p(\Omega)$ for all $p\ge 2$. In addition, by the moment formula
\eqref{E2:SecMom-Up},
\begin{align}\label{EH:LpBd}
\sup_{(t,x)\in
[0,T]\times\R}
\Norm{u(t,x)}_p^2
\le
2\:C^2 + \left(2\:C^2+\Vip^2\right)
\widehat{\calH}_p\left(T \right)
<+\infty,\quad\text{for all $T>0$,}
\end{align}
where $C=\sup_{x\in\R}|f(x)| = \sup_{(t,x) \in \R_+\times\R} J_0(t,x)$.
From this bound, one can easily derive that that $u\in 
C_{1/4-,1/2-}\left(\R_+^*\times\R\right)$, a.s.: see Remark 
\ref{RH:BddHolder} below.
We will extend this classical result to the case where $\mu$ can be
unbounded either locally, such as $\mu=\delta_0$, or at $\pm\infty$, such as
$\mu(\ud x) = e^{|x|^a}\ud x$, $a\in \:]1,2[\:$, or both.
However, for irregular initial conditions, H\"older continuity of $u$ will be 
obtained only on $\R_+^*\times\R$, and this continuity extends to 
$\R_+\times\R$ when the initial condition is continuous.

We need a set of initial data defined as 
follows:
\[
\calM_H^*(\R) :=  \left\{\mu(\ud x)=f(x)\ud x,\:\text{s.t. $\exists 
a\in \: ]1,2[\;$, $\sup_{x\in\R} |f(x)|e^{-|x|^a}<+\infty$} \right\}.
\]
Clearly, $\calM_H^*(\R)\subset \calM_H(\R)$, and $\calM_H^*(\R)$ includes all 
absolutely continuous measures whose density functions are bounded by functions 
of the type $c_1 e^{c_2 |x|^a}$ with
$c_1,c_2>0$ and $a\in\: ]1,2[\;$ (see Lemma \ref{LH:MHR}).

\begin{theorem}\label{T2:Holder}
Suppose that $\rho$ is Lipschitz continuous. Then the solution
$u(t,x)=J_0(t,x)+I(t,x)$ to \eqref{EH:WalshSI} has the following sample path
regularity:
\begin{enumerate}[(1)]
\item If $\mu\in\calM_H(\R)$, then $I\in
C_{\frac{1}{4}-,\frac{1}{2}-}\left(\R_+^*\times\R\right)$ a.s. 
Therefore, 
\[
u= J_0+I\in
C_{\frac{1}{4}-,\frac{1}{2}-}\left(\R_+^*\times\R\right),\;\;\text{a.s.}
\]
\item If $\mu\in\calM_H^*(\R)$,  then $I \in 
C_{\frac{1}{4}-,\frac{1}{2}-}\left(\R_+\times\R\right)$, a.s. If, in addition, $\mu(\ud x)=f(x)\ud x$, where $f$ is 
a continuous function, then 
\[
u\in C\left(\R_+\times\R\right) \cap C_{\frac{1}{4}-,\frac{1}{2}-}\left(\R_+^*\times\R\right),\quad \text{a.s.}
\]
If  $\mu\in\calM_H^*(\R)$ and, in addition, $\mu(\ud x)=f(x)\ud x$, where $f$ is an $\alpha$-H\"older continuous function, then
\[
u\in C_{\left(\frac{\alpha}{2}\wedge \frac{1}{4}\right)-,\left(\alpha\wedge 
\frac{1}{2}\right)-}
\left(\R_+\times\R\right) \cap C_{\frac{1}{4}-,\frac{1}{2}-}\left(\R_+^*\times\R\right),\quad \text{a.s.}
\]
\end{enumerate}
\end{theorem}
This theorem will be proved in Section \ref{ss:Holder}.

\begin{remark}\label{RH:Standard}
The standard approach (e.g., that is used in \cite[p.54 --55]{Dalang09Mini2},
\cite{SanSoleSarra99Holder}, \cite{Shiga94Two} and \cite{Walsh86}) for proving 
H\"older
continuity cannot be used to establish the above theorem. For instance, 
consider the case
where $\rho(u)=u$ and $\mu=\delta_0$.  The classical argument,
as presented in \cite[p.432]{Shiga94Two} (see also the proof of Proposition 
1.5 in \cite{BallyMilletSS95} and the proof of Corollary 3.4 in 
\cite{Walsh86}), uses Burkholder's inequality for $p>1$ and 
H\"older's inequality with $q=p/(p-1)$ to obtain
\begin{align*}
 \Norm{I(t,x)-I(t',x')}_{2p}^{2p}
\le &C_{p,T}
\left(\int_0^{t\vee t'}\hspace{-0.8em}\int_\R  \ud s \ud y \left(G_ \nu(t-s,x-y)
-G(t'-s,x'-y')\right)^2  \right)^{p/q}\\
&\times
\int_0^{t\vee t'}\hspace{-0.8em}\int_\R \ud s \ud y \left(G_ \nu(t-s,x-y)
-G(t'-s,x'-y')\right)^2\left(1+\Norm{u(s,y)}_{2p}^{2p}\right).
\end{align*}
However, by H\"older's inequality, 
\eqref{E2:SecMom-Delta} and \eqref{E:K},
\[\Norm{u(s,y)}_{2p}^2 \ge
\Norm{u(s,y)}_2^2 \ge G_{\nu/2}(s,y) \frac{1}{\sqrt{4\pi\nu s}}\:.\]
Therefore, $\Norm{u(s,y)}_{2p}^{2p}\ge C G_{\nu/(2p)}(s,y) s^{1/2-p}$. The
second term in the above bound is not $\ud s$--integrable in a neighborhood of 
$\{0\}\times\R$ unless $p<3/2$. Therefore, this classical argument does not 
apply in the presence of an irregular initial condition such as $\delta_0$.
\end{remark}

\begin{example}[Dirac delta initial data]\label{E2:Holder-Delta}
Suppose $\rho(u)= \lambda u$ with $\lambda \ne 0$.
If $\mu=\delta_0$, then neither $x\mapsto J_0(0,x)$ nor 
$x\mapsto \lim_{t\rightarrow 0_+}
\Norm{I(t,x)}_2$ is a continuous function. Indeed, this is clear for
$J_0(0,x) = \delta_0(x)$.
For $\lim_{t\rightarrow 0_+} \Norm{I(t,x)}_2$, by \eqref{E2:SecMom-Delta},
\[
\Norm{I(t,x)}_2^2=
\Norm{u(t,x)}_2^2 - J_0^2(t,x) =
\frac{\lambda^2}{2\nu}e^{\frac{\lambda^4
t}{4\nu}}\Phi\left(\lambda^2\sqrt{\frac{t}{2\nu}}\right) G_{\nu/2}(t,x)\;.
\]
Therefore, $\lim_{t\rightarrow 0_+}
\Norm{I(t,x)}_2^2$ equals $0$ if $x\ne 0$, and $+\infty$ if $x=0$.
(We note that $I(0,x)\equiv 0$ by definition.)
\end{example}

Example \ref{E2:Holder-Delta} suggests that $\Norm{I(t,x)}_2^2$ tends to 
$\frac{\lambda^2}{4\nu} \delta_0(x)$ as $t\rightarrow 0_+$ in the weak 
sense, i.e., 
\[
\lim_{t\rightarrow 0_+}\InPrd{\Norm{I(t,\cdot)}_2^2, \phi(\cdot)} =
\frac{\lambda^2}{4\nu}\phi(0),\quad\text{for all $\phi\in C_c^\infty(\R)$,}
\]
where $C_c^\infty(\R)$ denotes smooth functions with compact support.
Furthermore, the following proposition shows that the 
random field solution of \eqref{EH:WalshSI} satisfies the initial condition 
$u(0,\circ) =\mu$ in a weak sense.

\begin{proposition}\label{PH:Weak}
 For all $\phi\in C_c^\infty(\R)$ and $\mu\in\calM_H(\R)$,
\[
\lim_{t\rightarrow 0_+} 
\int_\R \ud x\: u(t,x) \phi(x) = \int_\R \mu (\ud x)\: \phi(x)
\quad\text{in $L^2(\Omega)$.}
\]
\end{proposition}
The proof of this proposition is presented in Section \ref{ss:Weak}.
The next proposition shows that $t\mapsto I(t,x)$ may be quite far from 
$\frac{1}{4}$--H\"older continuous at the origin, and in fact, the 
H\"older-exponent 
may be arbitrarily near $0$.

\begin{proposition}\label{P:Holder-Ubd}
Suppose $\rho(u)= \lambda u$ with $\lambda \ne 0$ and $\mu(\ud x)=
|x|^{-a}\,\ud x$ with $0<a\le 1$, so that
$J_0(0,x)= |x|^{-a}$ is not locally bounded. Fix $p \geq 2$. Then:

{\em (1)} If $a<1/2$, then for all $x \in \R$, $\lim_{t\rightarrow 0_+}\Norm{I(t,x)}_p\equiv 0$.

{\em (2)} There is $c>0$ such that for all $t>0$, $\Norm{I(t,0)}_p\ge c \:
t^{\frac{1-2a}{4}}$. 

\noindent In particular, when $\frac{1}{2}< a <1$, 
$\lim_{t\rightarrow 0_+} \Norm{I(t,0)}_p = +\infty$, and when $0<a<\frac{1}{2}$,
$t\mapsto I(t,0)$ from $\R_+$ to $L^p(\Omega)$ 
cannot be
smoother than $\frac{1-2a}{4}$--H\"older continuous (in this case 
$\frac{1-2a}{4}\in\:]0,1/4[\:$).
\end{proposition}

\begin{proof}
(1) By the moment bounds formulas \eqref{E2:SecMom-Up} and \eqref{E2:SecMom}, 
it 
suffices to consider second moment and  show that
$\lim_{t\rightarrow 0_+}
\Norm{I(t,x)}_2\equiv 0$.
For some constant $C_a>0$, the Fourier transform of $\mu$ is $C_a |x|^{-1 +a}$
(see \cite[Lemma 2 (a), p. 117]{Stein70Singular}), which is non-negative.
Hence Bochner's theorem (see, e.g., \cite[Theorem 1, 
p.152]{GelfandVilenkin64GF4}) implies that $\mu$, and therefore $x\mapsto
J_0(t,x)$, is non-negative definite. Such functions achieves their maximum at
the origin (see, e.g., \cite[Theorem 1, p. 152]{GelfandVilenkin64GF4}), and so
\[0< J_0(t,x) \le J_0(t,0)
 = \int_\R \ud y\: \frac{1}{|y|^a} G_\nu(t,y)
=2 \int_0^{+\infty}\ud y\: \frac{e^{-y^2/(2\nu t)}}{y^a \sqrt{2\pi\nu t}}\: .\]
Then by a change of variable and using Euler's integral (see \cite[5.2.1, 
p.136]{NIST2010}),
\begin{align}\label{E2_:Holder-Ubd-G}
 J_0(t,0)= 2 \int_0^{+\infty}\ud u\: \frac{e^{-u}}{(2\nu t u)^{a/2} 
\sqrt{2\pi\nu t}}
\frac{\sqrt{2\nu t}}{2\sqrt{u}} 
=\frac{\Gamma\left(\frac{1-a}{2}\right)}{\sqrt{\pi} (2\nu t)^{a/2}}\:,
\end{align}
where $\Gamma(x)$ is Euler's Gamma function \cite{NIST2010}.
By \eqref{E2:SecMom} and the above bound,
\begin{align*}
\Norm{I(t,x)}_2^2 &= \left(J_0^2 \star \calK\right)(t,x)
\le
\int_0^t\ud s\:  \left(\frac{\lambda^2}{\sqrt{4\pi\nu (t-s)}} +
\frac{\lambda^4}{2\nu}e^{\frac{\lambda^4 (t-s)}{4\nu}}\right) \frac{C}{s^a} .
\end{align*}
The integral converges if and
only if $a<1$.
Finally, using the Beta integral (see \cite[(5.12.1), p.142]{NIST2010})
\begin{align}
\label{E2:BetaInt}
\int_0^t \ud s\:  s^{\mu-1} (t-s)^{\nu-1} = 
t^{\mu+\nu-1}\frac{\Gamma(\mu)\Gamma(\nu)}{\Gamma(\mu+\nu)},
\quad \text{for $t>0$, $\mu>0$ and $\nu>0$},
\end{align}
we see that $\Norm{I(t,x)}_2^2\le C_1\: t^{1/2-a}+ C_2 \:t^{1-a}$, so
$\lim_{t\rightarrow 0_+}\Norm{I(t,x)}_2^2=0$ when $a<1/2$.

(2) Now consider the function $t\mapsto I(t,0)$ from $\R_+$ to 
$L^p(\Omega)$.
Since $(x-y)^2\le 2(x^2 + y^2)$, as in \eqref{E2_:Holder-Ubd-G},
we see that
\[
J_0(t,x) = \int_\R\ud y\: \frac{1}{|y|^a} G_\nu\left(t,x-y\right)
\ge
\frac{1}{\sqrt{2}}\exp\left(-\frac{x^2}{\nu t}\right)
\frac{\Gamma\left(\frac{1-a}{2}\right)}{\sqrt{\pi}} \frac{1}{(\nu t)^{a/2}}.
\]
Hence,
\[
J_0^2(t,x) \ge
C G_{\nu/2}\left(\frac{t}{2},x\right) t^{1/2-a}.
\]
Since $\calK(t,x)\ge G_{\nu/2}(t,x) \frac{\lambda^2}{\sqrt{4\pi\nu
t}}$ by \eqref{E:K},
\[\Norm{I(t,x)}_2^2
\ge
\frac{C'  \exp\left(-\frac{2 x^2}{\nu
t}\right)}{t} \int_0^t\ud s \:
s^{1/2-a} = C'' \exp\left(-\frac{2 x^2}{\nu
t}\right) t^{\frac{1-2a}{2}}.
\]
If $x=0$, then for all integers $p\ge 2$, since $I(0,x)\equiv
0$,
\[\Norm{I(t,0)- I(0,0)}_p^2\ge \Norm{I(t,0)}_2^2
\ge  C'' \:t^{\frac{1-2a}{2}}.
\]
When $0<a<1/2$, the function $t\mapsto I(t,0)$ from $\R_+$ to $ L^p(\Omega)$
cannot be smoother than $\eta$-H\"older
continuous at $t=0$ with $\eta= \frac{1-2a}{4} \in \:]0,1/4[\:$.
\end{proof}

\section{Proofs of the main results}\label{SH:Proofs}

   Establishing H\"older continuity relies on Kolmogorov's continuity theorem. 
We present a formulation of this result that is suitable for our purposes.

\subsection{Kolmogorov's continuity theorem}

For $x=(x_1,\dots,x_N)$ and $y=(y_1,\dots,y_N)$, define
\begin{align}\label{E2:metric-a}
\tau_{\alpha_1,\dots,\alpha_N}(x,y) := \sum_{i=1}^N
\left|x_i-y_i\right|^{\alpha_i}\;,\quad
\text{with $\alpha_1,\dots,\alpha_N\in \;]0,1]$.}
\end{align}
This defines a metric on $\R^N$ that is not induced by a norm except when 
$\alpha_i=1$ for $i=1,\dots,N$.
We refer the interested readers to \cite[Theorem
4.3]{Khoshnevisan09Mini} or \cite[Theorem 2.1, on p.
62]{RevuzYor99Continuous} for the isotropic case ($\alpha_1=\cdots
=\alpha_N$).
For the anisotropic case (where the $\alpha_i$ are not identical),
see \cite[Theorem 1.4.1,
p. 31]{Kunita90Flow} and \cite[Corollary A.3, p.
34]{DalangKhNualart07HittingAdditive}. We state a version (Proposition 
\ref{P2:LocHolder} below), which is a consequence of these references and is 
convenient for our purposes.

\begin{definition}(H\"older continuity)\label{D2:Holder}
A function $f:D\mapsto \R$ with $D\subseteq \R^N$ is said to be {\it locally 
(and uniformly) H\"older continuous} with
indices $\left(\alpha_1,\dots,\alpha_N\right)$ if for
all compact sets $K\subseteq D$, there exists a constant $A_K$ such that
for all $x,y \in K$, $\left|f(x)-f(y)\right|\le A_K \sum_{i=1}^N
|x_i-y_i|^{\alpha_i}$.
\end{definition}

\begin{proposition}\label{P2:LocHolder}
Let $\left\{X(t,x): \: (t,x)\in \R_+\times\R^d \:\right\}$ be a random field
indexed by $\R_+\times\R^d$. Suppose that there exist $d+1$ constants
$\alpha_i \in \;]0,1]$, $i=0, 1,\dots, d$, such that for all $p> 2(d+1)$  and
all $n>1$,
there is a constant $C_{p,n}$ such that 
\[\Norm{X(t,x)-X\left(s,y\right)}_p
\le C_{p,n}
\tau_{\alpha_0,\dots,\alpha_d}\left((t,x),\left(s,y\right)\right)
\]
for all $(t,x)$ and $(s,y)$ in $K_n:= \left[1/n,n\right] \times
[-n,n]^d$,
where the metric $\tau_{\alpha_0,\dots,\alpha_d}$ is defined in
\eqref{E2:metric-a} with $N=d+1$. Then $X$ has a modification which is locally
H\"older continuous in $\R_+^*\times\R^d$ with indices $\left(\beta 
\alpha_0,\dots,\beta \alpha_d\right)$, for
all $\beta\in \;\left]0,1 \right[$.
If the compact sets $K_n$ can be taken to be $\left[0,n\right]
\times [-n,n]^d$, then the same local H\"older continuity of $X$ extends to
$\R_+\times\R^d$.
\end{proposition}

\subsection{Moment estimates}\label{ss:Holder}

   The main moment estimate that is needed for this proof is the following.
	
\begin{proposition}\label{rd:moment}
Fix $\Vip \in\R$ and $\mu\in\calM_H(\R)$.

{\em (1)} For all $p \geq 2$ and $n>1$, there is a constant $C_{n,p}$ such that 
for
all $t,t'\in [1/n,n]$ and $x,x'\in [-n,n]$,
\begin{equation}\label{rd4.1}
   \Vert I(t,x) - I(t',x')\Vert_p \leq C_{n,p} \left(\vert t - 
t'\vert^{\frac{1}{4}} + \vert x - x'\vert^{\frac{1}{2}} \right).
\end{equation}

{\em (2)} If, in addition, $\mu\in\calM_H^*(\R)$, then there exists a constant 
$C_{n,p}^*$ such 
that for all $(t,x)$, $(t',x')\in [0,n]\times [-n,n]$,  \eqref{rd4.1} holds 
with $C_{n,p}$ replaced by $C_{n,p}^*$.
\end{proposition}

   The proof of this proposition will be given at the end of this section. We 
note that by Proposition \ref{P:Holder-Ubd}, the conclusion in part (2) above 
is not valid for all $\mu\in\calM_H(\R)$.
	
	Assuming Proposition \ref{rd:moment}, we now prove Theorem \ref{T2:Holder}.
	
\begin{proof}[Proof of Theorem \ref{T2:Holder}]
By Lemma \ref{LH:J0Cont}, we only need to establish the H\"older-continuity 
statements for $I$ instead of $u$. Part (1) (respectively (2)) follows from 
Proposition \ref{rd:moment}(1) (respectively Proposition \ref{rd:moment}(2)) 
and Proposition \ref{P2:LocHolder}. This proves Theorem \ref{T2:Holder}.
\end{proof}

   The next two propositions are needed to establish Proposition 
\ref{rd:moment}. 
	
\begin{proposition} \label{P2:Holder-I}
Given $\Vip \in\R$ and $\mu\in\calM_H(\R)$,
let $J_0^*(t,x)= \left(|\mu|*G_\nu(t,\cdot)\right)(x)$ and
$h(t,x)=\Vip^2+2 \left[J_0^*(t,x)\right]^2$.
Then we have:\\
{\em (1)} For all $n>1$, there exist constants $C_{n,i}$, $i=1,3,5$, such 
that for all $t,t'\in [1/n,n]$, with $t< t'$, and $x,x'\in [-n,n]$,
\begin{gather}\label{E2:H1}
 \iint_{\left[0,t\right]\times\R}\ud s\ud y\:
h(s,y)
\left(
G_\nu\left(t-s,x-y\right)-G_\nu(t'-s,x-y)
\right)^2
 \le C_{n,1} \sqrt{t'-t},\\
\label{E2:H3}
 \iint_{[0,t]\times\R} \ud s\ud y\: h(s,y)
\left(
G_\nu\left(t-s,x-y\right)-G_\nu(t-s,x'-y)
\right)^2
 \le C_{n,3} \left|x-x'\right|,\\
\label{E2:H5}
\iint_{[t,t']\times\R}\ud s\ud y\: h(s,y)
G_\nu^2 (t'-s,x'-y)  \le C_{n,5} \sqrt{t'-t}\:.
\end{gather}
{\em (2)} If, in addition, $\mu\in\calM_H^*(\R)$, then 
there  exist constants $C_{n,i}^*$, $i=1,3,5$, such 
that for all $(t,x),(t',x')\in [0,n]\times 
[-n,n]$, \eqref{E2:H1}--\eqref{E2:H5} hold  with $C_{n,i}$ replaced by 
$C_{n,i}^*$, $i=1,3,5$.
\end{proposition}

\begin{proposition} \label{P2:Holder-II}
Given $\Vip \in\R$ and $\mu\in\calM_H(\R)$,
let $J_0^*(t,x)= \left(|\mu|*G_\nu(t,\cdot)\right)(x)$.
Then:\\
{\em (1)} For all $n>1$, there
exist three constants
\begin{gather}\label{E2:Cn246}
 C_{n,2} = \frac{\sqrt{\pi n}}{\sqrt{4\nu}} C_{n,1},\quad
C_{n,4} = \frac{\sqrt{\pi n}}{\sqrt{4\nu}} C_{n,3},\quad\text{and}\quad
C_{n,6} = \frac{\sqrt{\pi n}}{\sqrt{4\nu}} C_{n,5},
\end{gather}
such that for all $t,t'\in [1/n,n]$ with $t< t'$ and $x,x'\in [-n,n]$,
\begin{gather}
\label{E2:H2}
\left|\left(\left(\Vip^2+2\left|J_0^*\right|^2\right) \star G_\nu^2 \star
\left(G_\nu(\cdot,\circ)-G_\nu(\cdot+t'-t,\circ)\right)^2 \right)(t,x)\right|
\le C_{n,2}\sqrt{t'-t},\\
\label{E2:H4}
\left|\left(\left(\Vip^2+2\left|J_0^*\right|^2\right) \star G_\nu^2 \star
\left(G_\nu(\cdot,\circ)-G_\nu(\cdot,\circ+x'-x)\right)^2 \right)(t,x)\right|
\le C_{n,4}|x'-x|,\\
\label{E2:H6}
\iint_{[t,t']\times\R}
\ud s \ud y \left(\left(\Vip^2+2\left|J_0^*\right|^2\right) \star
G_\nu^2 \right) \left(s,y\right)
G_\nu^2(t'-s,x'-y) 
\le C_{n,6}\sqrt{t'-t}.
\end{gather}
{\em (2)} If, in addition, $\mu\in\calM_H^*(\R)$, then 
there  exist constants 
\begin{gather*}
 C_{n,2}^* = \frac{\sqrt{n}}{\sqrt{\pi\nu}} C_{n,1}^*,\quad
C_{n,4}^* = \frac{\sqrt{n}}{\sqrt{\pi\nu}} C_{n,3}^*,\quad\text{and}\quad
C_{n,6}^* = \frac{\sqrt{n}}{\sqrt{\pi\nu}} C_{n,5}^*,
\end{gather*}
 such  that for all $(t,x),(t',x')\in [0,n]\times 
[-n,n]$, \eqref{E2:H2}--\eqref{E2:H6} hold  with $C_{n,i}$ replaced by 
$C_{n,i}^*$, $i=2,4,6$.
\end{proposition}

The proofs of these two propositions are given in the Sections \ref{sec4.3} and 
\ref{sec4.4}. Assuming Propositions  \ref{P2:Holder-I} and \ref{P2:Holder-II}, 
we now prove Proposition \ref{rd:moment}.

\begin{proof}[Proof of Proposition \ref{rd:moment}]
We first prove part (1). 
Without loss of generality, assume  that $\mu\ge 0$. Otherwise, we simply 
replace $\mu$ in the following arguments by
$|\mu|$. Fix $n>1$. By parts (1) of Propositions \ref{P2:Holder-I} and
\ref{P2:Holder-II}, there exist $C_{n,i}>0$ for
$i=1,\dots,6$ such that for all $(t,x)$ and $\left(t',x'\right)\in [1/n,n]\times
[-n,n]$ with $t'>t$, the six inequalities in Propositions  \ref{P2:Holder-I}
and \ref{P2:Holder-II} hold. By \eqref{EH:LinGrow} and Lemma \ref{LH:Lp},
for all even integers $p\geq 2$,
\begin{align*}
  \Norm{I(t,x)-I\left(t',x'\right)}_p^p
\le  2^{p-1} z_p^p \Lip_\rho^p I_1\left(t,t',x,x'\right)^{p/2}
  +
  2^{p-1} z_p^p  \Lip_\rho^p I_2\left(t,t'\:;\:x'\right)^{p/2}\;,
\end{align*}
where
\begin{gather}\label{EH:LH1}
I_1\left(t,t',x,x'\right)=\iint_{\left[0,t\right]\times\R}\ud s \ud y\:
  \left(G_\nu\left(t-s,x-y\right)-G_\nu(t'-s,x'-y) \right)^2 \left[\Vip^2
+\Norm{u\left(s,y\right)}_p^2\right],\\
I_2\left(t,t'\:;\:x'\right)
  = \iint_{\left[t,t'\right]\times\R}\ud s \ud y\:
G_\nu^2\left(t'-s,x'-y\right)\left(\Vip^2+\Norm{u\left(s,y\right)}
_p^2\right).
\label{EH:LH2}
\end{gather}
By the subadditivity of $x\mapsto |x|^{2/p}$ and
since $2^{2(p-1)/p}\le 4$,
\[
\Norm{I(t,x)-I\left(t',x'\right)}_p^2 \le
4 z_p^2 \Lip_\rho^2 \left[ 
I_1\left(t,t',x,x'\right)+I_2\left(t,t'\:;\:x'\right)
\right].
\]
Notice that
\[\calK \left(t,x;\nu, \lambda\right)=
\Upsilon\left(t;\nu,\lambda\right) \;
G_{\nu}^2(t,x),\quad\text{with $\Upsilon(t;\nu,\lambda)=
 \lambda^2+\lambda^4 \sqrt{\frac{\pi t}{\nu}}
e^{\frac{\lambda^4 t}{4\nu}}
\Phi\left(\lambda^2 \sqrt{\frac{t}{2\nu}}\;\right)$.}
\]
Denote $\Upsilon_*(t) :=
\Upsilon\left(t\;;\;\nu, \:a_{p,\Vip}z_p\Lip_\rho\right)<+\infty$, for
all $t\in\R_+$.
Clearly, $\Upsilon_*(t)\le \Upsilon_*(n)$ for $t\le n$.
Hence, it follows from \eqref{E2:SecMom-Up} and
\eqref{E2:H} that
\begin{align}\label{EH:BddMom}
\Norm{u\left(s,y\right)}_p^2
\le
2\: J_0^2\left(s,y\right)+
\Upsilon_*(n)\left(\left(\Vip^2+2 \: J_0^2\right)\star
G_\nu^2\right)(s,y),\quad\text{for $s\le t\le n$.}
\end{align}
We shall use this bound in order to estimate $I_1$
and $I_2$.

We first consider the case where $x=x'$. Set $h=t'-t$. Then
\begin{align*}
I_1(t,t',x,x) \le& \left(\left(\Vip^2+2J_0^2\right)\star
\left(G_\nu(\cdot,\circ)-G_\nu(\cdot+h,\circ)\right)^2\right)(t,x)\\
&+ \Upsilon_*(n) \left(\left(\Vip^2+2J_0^2\right)\star G_\nu^2
\star
\left(G_\nu(\cdot,\circ)-G_\nu(\cdot+h,\circ)\right)^2\right)(t,x).
\end{align*}
By parts (1) of Propositions \ref{P2:Holder-I} and \ref{P2:Holder-II}, 
\[
I_1(t,t',x,x)\le \left(C_{n,1}+\Upsilon_*(n)  C_{n,2} \right)\:
|h|^{1/2}.
\]
Similarly, we have that
\begin{align*}
I_2\left(t,t'\:;\:x'\right)\le \left(C_{n,5} + \Upsilon_*(n) C_{n,6}\right)\:
|h|^{1/2}.
\end{align*}
Hence, for all $x\in [-n,n]$ and $1/n\le t<t'\le n$,
\begin{align}\label{E2_:CaseIHold}
\Norm{I(t,x)-I(t',x)}_p^2
\le&\:
4 z_p^2 \Lip_\rho^2
\left(C_{n,1}+C_{n,5}+\Upsilon_*(n)\left(C_{n,2}+C_{n,6}\right)\right) \:
\left|t'-t \right|^{1/2}\:.
\end{align}

Now consider the case where $t=t'\ge 1/n$. Denote $\zeta=x'-x$.
In this case, $I_2=0$. By \eqref{EH:BddMom} above and parts (1) of
Propositions  \ref{P2:Holder-I} and \ref{P2:Holder-II}, 
\[\Norm{I(t,x)-I\left(t,x'\right)}_p^2\le 4 z_p^2 \Lip_\rho^2 \: \left[ C_{n,3}+
\Upsilon_*(n)
C_{n,4}\right]\:|\zeta|.
\]
Combining this with \eqref{E2_:CaseIHold}, we see that
\[\Norm{I(t,x)-I\left(t',x'\right)}_p^2
\le
\widetilde{C}_{p,n} \left(\left|t'-t\right|^{1/2}+
\left|x'-x\right|\right),
\]
for all $1/n\le t <t'\le n$, $x,x'\in
[-n,n]$,  where $\widetilde{C}_{p,n}$ is a finite constant.
This proves (1).

The conclusion in part (2) can be proved in the same way by applying parts (2) 
of 
Propositions \ref{P2:Holder-I} and \ref{P2:Holder-II} below instead of parts 
(1). We simply replace all $C_{n,i}$
above by $C_{n,i}^*$ for $i=1,\dots,6$. The remaining statements follow 
immediately. This completes the proof of Proposition \ref{rd:moment}.
\end{proof}

\begin{remark}[Case of bounded initial data]\label{RH:BddHolder}
In the case where the initial data is bounded: $\mu(\ud x)= f(x)\ud x$, 
where $f$ is a bounded function such that $|f(x)|\le C$, 
the conclusions of Proposition \ref{rd:moment} follow from the following 
standard 
(and much simpler) argument:
By \eqref{EH:LpBd}, for $0\le t\le 
t'\le T$, and $x,x'\in\R$
\[
I_1(t,t',x,x') \le A_T \iint_{[0,t']\times\R}  \ud s \ud y\;
\left(G_\nu(t-s,x-y)-G_\nu(t'-s,x'-y)\right)^2,
\]
where $I_1(t,t',x,x')$ is defined in \eqref{EH:LH1} and $A_T$ is a finite 
constant.
Then by Proposition \ref{PH:G}, for some constant $C'>0$ depending only 
on 
$\nu$,
\[
I_1\left(t,t',x,x'\right) \le A_T C' \left(\left|x-x'\right|
+\sqrt{\left|t'-t\right|}\right)\;.
\]
Similarly, $I_2\left(t,t',x,x'\right)$, defined in 
\eqref{EH:LH2}, is bounded by $ A_T C'
\sqrt{\left|t'-t\right|}$ with
the same constants $A_T$ and $C'$.
Therefore, 
\begin{align*}
\Norm{I(t,x)-I\left(t',x'\right)}_p^2
\le
4 z_p^2 A_T C' \left( \left|x-x'\right|+ |t-t'|^{1/2}\right)\;,
\end{align*}
for all $0\le t\le t \le  T$ and $x,x'\in\R$.
The H\"older continuity follows from Proposition
\ref{P2:LocHolder}.
\end{remark}

\subsection{Proofs of part (1)  of the Propositions \ref{P2:Holder-I} and 
\ref{P2:Holder-II}} \label{sec4.3}

\begin{lemma}\label{LH:GG2G}
For all $L>0$, $\beta\in \;]0,1[$\:, $t>0$, $x\in\R$, $\nu>0$, and $h$ with 
$|h|\le \beta L$,
we have that
\begin{multline*}
\left|G_\nu(t,x+h) - G_\nu(t,x)
\right| \\
\le
 |h| \left(\frac{C}{\sqrt{2\nu
t}}+\frac{1}{(1-\beta)L}\right)
\left[
G_{\nu}(t,x) +
e^{\frac{3L^2}{2\nu t}} \left\{G_\nu \left(t,x-2L\;\right)
+
G_\nu \left(t,x+2L\;\right)\right\}
\right]
\end{multline*}
and
\begin{multline*}
\left|G_\nu(t,x+h)+G_\nu(t,x-h)
-2 G_\nu(t,x)
\right|  \\ 
\le
2 |h| \left(\frac{C}{\sqrt{2\nu
t}}+\frac{1}{(1-\beta)L}\right)
\left[
G_{\nu}(t,x) +
e^{\frac{3L^2}{2\nu t}} \left\{G_\nu \left(t,x-2L\;\right)
+
G_\nu \left(t,x+2L\;\right)\right\}
\right],
\end{multline*}
where $C:=\sup_{x\in\R} \frac{1}{|x|} |e^{-x^2/2}-1| \approx
0.451256$.
\end{lemma}

\begin{proof}
Fix $L>0$ and $\beta\in \;]0,1[$. Assume that $|h|\le
\beta L$. Define
\begin{gather*}
f(t,x,h) = G_\nu(t,x+h) + G_\nu(t,x-h) -2  G_\nu(t,x),\\
I(t,x,h) =\begin{cases}
h^{-1} \:G_\nu^{-1}(t,x-L) \big[G_\nu(t,x+h)-G_\nu(t,x)\big]& \text{if $x\ge
0$,}
\cr
h^{-1} \:G_\nu^{-1}(t,x+L) \big[G_\nu(t,x+h)-G_\nu(t,x)\big]& \text{if $x\le
0$.}
\end{cases}
\end{gather*}
Clearly,
\begin{align}\label{E2_:GG2G}
\left|\frac{f(t,x,h)}{h\:\left(G_\nu(t,x+L)+G_\nu(t,x-L)\right)}\right| \le
\left|I(t,x,h)\right| +
\left|I(t,x,-h)\right|\:.
\end{align}
We will bound $\left|I(t,x,h)\right|$ for $-\beta L \le h \le
\beta L$.
If $x\ge 0$, then 
\[
I(t,x,h)=\frac{1}{h}(\:e^{-\frac{(x+h)^2}{2\nu
t}+\frac{(x-L)^2}{2\nu t}
}-
e^{-\frac{x^2}{2\nu t}+\frac{(x-L)^2}{2\nu t}}\:
),\]
and so
\[
\frac{\partial }{\partial x} I(t,x,h) =
-\frac{1}{\nu t}e^{-\frac{(x+h)^2}{2\nu t}+\frac{(x-L)^2}{2\nu t}
} - \frac{L}{\nu t} \:I(t,x,h).
\]
Hence,
\[\left|I(t,x,h)\right|
\le
\int_0^x (\nu t)^{-1} e^{-\frac{(y+h)^2}{2\nu t}+\frac{(y-L)^2}{2\nu t}}\ud y +
\frac{L}{\nu t} \int_0^x \left|I(t,y,h)\right|\ud
y+\left|I(t,0,h)\right|.
\]
Let $C$ be the constant defined in the proposition. Then
\[
\left|I(t,0,h)\right| \le \frac{C}{\sqrt{2\nu t}}
e^{\frac{L^2}{2\nu t}},
\quad \text{for all $h\in\R$.}
\]
Since $|h|\le \beta L$,
\begin{align*}
\int_0^x \frac{1}{\nu t}e^{-\frac{(y+h)^2}{2\nu t}+\frac{(y-L)^2}{2\nu
t}}\ud y   \le
\int_0^\infty \frac{1}{\nu t}e^{-\frac{(y+h)^2}{2\nu
t}+\frac{(y-L)^2}{2\nu t}}\ud y
= \frac{e^{\frac{L^2-h^2}{2\nu
t}}}{L+h}
\le \frac{e^{\frac{L^2}{2\nu t}}}{(1-\beta)L}.
\end{align*}
Therefore,
\[
\left|I(t,x,h)\right|
\le
C_{t, L,\beta}
+ \frac{L}{\nu t} \int_0^x
\left|I(t,y,h)\right| \ud y,\quad \text{with $
C_{t, L,\beta} :=\left(\frac{C}{\sqrt{2\nu t}}
+\frac{1}{(1-\beta)L}\right)e^{\frac{L^2}{2\nu t}}$.}
\]
Apply Bellman-Gronwall's lemma (see \cite[Lemma 12.2.2]{Kuo05Introduction}) to
get
\[
\left|I(t,x,h)\right|  \le
 C_{t, L,\beta}\:e^{\frac{L x}{\nu t}} =
C_{t, L,\beta} \: e^{\frac{L |x|}{\nu t}},
\]
and so, by definition of  $I(t,x,h)$,
\begin{align}\label{E2_:gg2g}
\left|G_\nu(t,x+h)-G_\nu(t,x)\right|
\le C_{t,L,\beta} |h| \left(G_\nu(t,x+L)+G_\nu(t,x-L)\right) e^{\frac{L
|x|}{\nu t}}.
\end{align}
By symmetry, for $x\le 0$, we get the same bound for $|I(t,x,h)|$.
Hence, from \eqref{E2_:GG2G},
\begin{align}\label{E2_:gg3g}
\left|f(t,x,h)\right| \le 2 C_{t, L,\beta}
|h| \left(G_{\nu}(t,x+L)+G_{\nu}(t,x-L)\right) \exp\left(\frac{L|x|}{\nu
t}\right).
\end{align}
Finally, some calculations show that
\begin{align*}
\big(G_{\nu}(t,x+L)+&G_{\nu}(t,x-L)\big) e^{\frac{L|x|}{\nu t}}\\
= & G_\nu(t,x) e^{-\frac{L^2}{2\nu t}} +
G_\nu(t,x-2L) e^{\frac{3L^2}{2\nu t}} \Indt{x\ge 0}
+
G_\nu(t,x+2L) e^{\frac{3L^2}{2\nu t}} \Indt{x\le 0}\\
\le&
 G_\nu(t,x) e^{-\frac{L^2}{2\nu t}} +
\bigg(G_\nu(t,x-2L)+G_\nu(t,x+2L)\bigg) e^{\frac{3L^2}{2\nu t}}\;.
\end{align*}
The desired conclusions now follow from \eqref{E2_:gg2g} and \eqref{E2_:gg3g}.
\end{proof}

\begin{proof}[Proof of Proposition \ref{P2:Holder-I} (1)]
Assume that $\Vip=0$. Set $\bar{z}=(z_1+z_2)/2$. Set
\[
I(t,x;t',x') = \iint_{[0,t]\times\R} \ud s\ud y
\left[J_0^*\left(s,y\right)\right]^2
\left(G_\nu\left(t-s,x-y\right)-G_\nu(t'-s,x'-y) \right)^2.
\]
Write $\left[J_0^*\left(s,y\right)\right]^2$ as a double integral
and then use Lemma \ref{LH:GG} to get
\begin{equation}\label{E2:Itxtx}
 \begin{aligned}
 I(t,x;t',x')=& \int_0^t\ud s \iint_{\R^2} |\mu|(\ud z_1)
|\mu|(\ud z_2) \: G_{2\nu}(s,z_1-z_2)\\
&\times \int_\R  \ud y\: G_{\nu/2}\left(s,y-\bar{z}\right)
\left(
G_\nu\left(t-s,x-y\right)-G_\nu(t'-s,x'-y)
\right)^2. 
 \end{aligned}
\end{equation}
In the following, we use $\int\ud y\: G(G-G)^2$ to denote the $\ud y$--integral
 in \eqref{E2:Itxtx}.
Expand $(G-G)^2 = G^2 -2 G G + G^2$ and apply Lemma \ref{LH:GG} to each term:
\begin{align*}
(
G_\nu&\left(t-s,x-y\right)-G_\nu(t'-s,x'-y)
)^2
\\
=&\frac{1}{\sqrt{4\pi\nu(t-s)}} G_{\nu/2}\left(t-s,x-y\right)
+\frac{1}{\sqrt{4\pi\nu(t'-s)}}
G_{\nu/2}\left(t'-s,x'-y\right)\\
&-
2G_{2\nu}\left(\frac{t+t'}{2}-s,x-x'\right) G_{\nu/2}\left(\frac{2(t-s)(t'-s)}{
t+t'-2s},y-\frac{(t-s)x'+(t'-s)x}{t+t'-2s} \right).
\end{align*}
Then integrate over $y$ using the semigroup
property of the heat kernel:
\begin{align}\notag
\int_\R \ud y& \: G_{\nu/2}(s,y-\bar{z}) 
\left(G_\nu(t-s,x-y)-G_\nu(t'-s,x'-y)\right)^2
 \\ \label{E2:GGD}
=\:&\frac{1}{\sqrt{4\pi\nu(t-s)}} 
G_{\nu/2}\left(t,x-\bar{z}\right)+\frac{1}{\sqrt{4\pi\nu(t'-s)}}
G_{\nu/2}\left(t',x'-\bar{z}\right)\\ \notag
&-2\:G_{2\nu}\left(\frac{t+t'}{2}-s,x-x'\right)G_{\nu/2}\left(
\frac{2(t-s)(t'-s)}{t+t'-2s}+s,
\frac{(t-s)x'+(t'-s)x}{t+t'-2s} -\bar{z}
\right).
\end{align}

{\vspace{1em}\noindent\bf Property \eqref{E2:H1}.~~} Set $x=x'$ in 
\eqref{E2:Itxtx} and let 
$h=t'-t$. 
Then $\frac{2(t-s)(t'-s)}{t+t'-2s}+s = t+\frac{(t-s)h}{2(t-s)+h}$
and \eqref{E2:GGD} becomes
\begin{align*}
\int\ud y\: G(G-G)^2
& =
\left[\frac{1}{\left(4\pi\nu 
(t-s)\right)^{\frac{1}{2}}}+\frac{1}{\left(4\pi\nu 
(t'-s)\right)^{\frac{1}{2}}}-
\frac{1}{\left(\pi 
\nu\left(\frac{t+t'}{2}-s\right)\right)^{\frac{1}{2}}}\right]
G_{\nu/2}\left(t,x-\bar{z}\right)\\
&\qquad
+ \frac{1}{\sqrt{4\pi\nu(t'-s)}}\left(
\frac{G_{\nu/2}\left(t',
x-\bar{z}\right)}{G_{\nu/2}\left(t,x-\bar{z}\right)} -1
\right)G_{\nu/2}\left(t,x-\bar{z}\right)\\
&\qquad
- \frac{1}{\sqrt{\pi \nu\left(\frac{t+t'}{2}-s\right) }}
\left(
\frac{G_{\nu/2}\left(t+\frac{(t-s)h}{2(t-s)+h},
x-\bar{z}\right)}{G_{\nu/2}\left(t,x-\bar{z}\right)} -1
\right)G_{\nu/2}\left(t,x-\bar{z}\right)\\
&:= I_1 + I_2 - I_3\;.
\end{align*}
We first consider $I_2$. Because $1/n\le t\le t'\le n$, 
we have that $h\in[0,n^2 t]$, so by Lemma \ref{L2:gtxr}, we find after 
simplification that
\begin{gather*}
\left|I_2\right| \le\frac{3\sqrt{1+n^2}}{4\sqrt{\pi\nu t(t'-s)}}
G_{\frac{\nu (1+n^2)}{2}}\left(t,x-\bar{z}\right)\sqrt{h},
\end{gather*}
and so
\begin{gather*}
\int_0^t \ud s\:G_{2\nu} (s,z_1-z_2) |I_2| 
\le \sqrt{h}\:
\int_0^t \ud s\:  \frac{3\sqrt{1+n^2}}{4\sqrt{\pi \nu t(t'-s)}}
G_{\nu(1+n^2)/2}\left(t,x-\bar{z}\right)
G_{2\nu}
(s,z_1-z_2) .
\end{gather*}
Apply Lemma \ref{LH:Split} to
$G_{\nu\left(1+n^2\right)/2}\left(\cdots\right) G_{2\nu}(\cdots)$ and
integrate over $\ud z_1\ud z_2$ to get
\[
\iint_{\R^2}|\mu|(\ud z_1)|\mu|(\ud z_2)\int_0^t \ud s \: G_{2\nu}|I_2|
\le
\frac{3\left(1+n^2\right) \sqrt{h}}{2\sqrt{\pi\nu}}\:
\left(|\mu|*G_{2\nu\left(1+n^2\right)}(t,\cdot)\right)^2(x)
 \int_0^t \frac{\ud s}{\sqrt{s(t'-s)}}.
\]
By the Beta integral (see \eqref{E2:BetaInt}), 
the $\ud
s$--integral is less than or equal to $\pi$. So
\begin{gather}
\iint_{\R^2}|\mu|(\ud z_1)|\mu|(\ud z_2)\int_0^t\ud s \: G_{2\nu}(\cdots)
|I_2| 
\le
\frac{3 \left(1+n^2\right) \sqrt{\pi}}{2\sqrt{\nu}}\:
\left(|\mu|*G_{2\nu\left(1+n^2\right)}(t,\cdot)\right)^2
(x) \sqrt{h}.
\label{E2:HolderT-I0}
\end{gather}

As for $I_3$, notice that since $s\in [0,t]$, $\frac{(t-s)h}{2(t-s)+h} \le
\frac{t h}{h} \le n^2
t$ for all $h\ge 0$.
Apply Lemma \ref{L2:gtxr} with  $r=
\frac{(t-s)h}{2(t-s)+h}$ to obtain that
\[\left|
\frac{G_{\nu/2}\left(t+\frac{(t-s)h}{2(t-s)+h},
x-\bar{z}\right)}{G_{\nu/2}\left(t,x-\bar{z}\right)} -1
\right|
\le
\frac{3}{2\sqrt{2}}\exp\left(\frac{n^2(x-\bar{z})^2}{\nu
t\left(1+n^2\right)}\right)
 \frac{\sqrt{h}}{\sqrt{t}},\quad\text{for all $h\ge 0$,}
\]
where we have used the inequality $\frac{(t-s)h}{2(t-s) +h}
\le \frac{h}{2}$.
Multiplying out the exponentials, we obtain
\[\left|I_3\right|
\le \frac{3\sqrt{1+n^2}}{2\sqrt{2\pi\nu t(t-s)}}
G_{\nu (1+n^2)/2}\left(t,x-\bar{z}\right)\:\sqrt{h}.
\]
Then by the same arguments as for $I_2$, we have that
\begin{gather*}
\iint_{\R^2}|\mu|(\ud z_1)|\mu|(\ud z_2)\int_0^t\ud s\: G_{2\nu}
 |I_3|  =\frac{3 \left(1+n^2\right) \sqrt{\pi}}{\sqrt{2\nu}}\:
\left(|\mu|*G_{2\nu\left(1+n^2\right)}(t,\cdot)\right)^2
(x)\sqrt{h}.
\end{gather*}

Now let us  consider $I_1$. Apply Lemma \ref{LH:Split} to
$G_{2\nu}\left(s,z_1-z_2\right) G_{\nu/2}\left(t,x-\bar{z}\right)$ to get
\begin{align*}
\int_0^t \ud s\: G_{2\nu}(s,z_1-z_2) |I_1|\le&
\frac{\sqrt{t}}{\sqrt{\pi\nu}}
 G_{2\nu}(t,x-z_1)
 G_{2\nu}(t,x-z_2)
\\ &\times \int_0^t
\ud s
\left|\left(s(t-s)\right)^{-\frac{1}{2}}+\left(s(t'-s)\right)^{-\frac{1}{2}}
- 2\left(s\left[\frac{t+t'}{2}-s\right]\right)^{-\frac{1}{2}}\right|.
\end{align*}
The integrand is bounded by
\[
\left|\left(s(t-s)\right)^{-\frac{1}{2}}-
\left(s\left[\frac{t+t'}{2}-s\right]\right)^{-\frac{1}{2}}\right|
+
\left|\left(s(t'-s)\right)^{-\frac{1}{2}}-
\left(s\left[\frac{t+t'}{2}-s\right]\right)^{-\frac{1}{2}}\right|.
\]
Taking into account the signs of the increment, this is equal to
$\left[s(t-s)\right]^{-1/2}-\left[s(t'-s)\right]^{-1/2}$.
Integrate the r.h.s.  of the above inequality using the formula
$\int_0^t \frac{\ud s}{\sqrt{s(t'-s)}}= 2
\arctan\left(\frac{\sqrt{t}}{\sqrt{t'-t}}\right)$ for all $t'>t\ge
0$ to find that
\[
\int_0^t\ud s 
\left|\left(s(t-s)\right)^{-\frac{1}{2}}+\left(s(t'-s)\right)^{-\frac{1}{2}}-
2\left(s\left[\frac{t+t'}{2}-s\right]\right)^{-\frac{1}{2}}\right| \le
\pi - 2\arctan\left(\sqrt{t/h}\right).
\]
It is an elementary calculus exercise to show that the function
$f(x):= x \left(\pi -2 \arctan \left(x\right)\right)$ for $x\ge
0$
is non-negative and bounded from above, and $f(x)\le
\lim_{x\rightarrow+\infty} f(x)= 2$.
Hence, $\pi -2\arctan\left(\sqrt{t/h}\right)
\le 2\sqrt{h/t}$.
Therefore,
\begin{align}\label{E2:HolderT-I2}
\iint_{\R^2}|\mu|(\ud z_1)|\mu|(\ud z_2)\int_0^t\ud s\:
 G_{2\nu} (s,z_1-z_2) |I_1| 
\le
\frac{2\sqrt{h}}{\sqrt{\pi \nu}}
\left(|\mu|*G_{2\nu}(t,\cdot)\right)^2 (x).
\end{align}

We conclude from \eqref{E2:HolderT-I0}--\eqref{E2:HolderT-I2} that for all
$(t,x), (t',x)\in [1/n,n]\times[-n,n]$
with $t'>t$,
\[
I(t,x;t',x)
\le \left(
C^{\star}_{\nu}\left(|\mu|*G_{2\nu}(t,\cdot)\right)^2 (x)
+C^*_{n,\nu} \left(|\mu|*G_{2\nu\left(1+n^2\right)}(t,\cdot)\right)^2 (x)
\right) \:\sqrt{h},
\]
where \[
C^{\star}_{\nu} = \frac{2}{\sqrt{\pi\nu}},\quad\text{and}\quad
C^*_{n,\nu}
:=\frac{3\sqrt{\pi}\left(1+\sqrt{2}\right)\left(1+n^2\right)}{2\sqrt{\nu}}.\]
As for the contribution of the constant $\Vip$, it corresponds to the initial
data $\mu(\ud x) \equiv \Vip\ud x$ and we apply  
Proposition \ref{PH:G}.
Finally, by the smoothing effect of the heat kernel (Lemma
\ref{LH:J0Cont}),
we can choose the following constant
\begin{align*}
C_{n,1} =
\Vip^2 \frac{\sqrt{2}-1}{\sqrt{\pi \nu}}
+
\sup
2  \left(
C^\star_{\nu}\left(|\mu|*G_{2\nu}(s,\cdot)\right)^2 (y)
+C^*_{n,\nu} \left(|\mu|*G_{2\nu\left(1+n^2\right)}(s,\cdot)\right)^2 (y)
\right),
\end{align*}
for \eqref{E2:H1}, where the supremum is over $(s,y)\in
[1/n,n]\times[-n,n]$. This proves \eqref{E2:H1}.

{\vspace{1em}\noindent\bf Property \eqref{E2:H3}.~~}
Set $t=t'$ in \eqref{E2:Itxtx} and $\bar{x}=\frac{x+x'}{2}$. Consider
the integral in \eqref{E2:Itxtx}
\[\int_0^t \ud s\:G_{2\nu}(s,z_1-z_2) \int\ud y\: G(G-G)^2,
\]
which is denoted by $\int\ud s\:G\int \ud y\: G(G-G)^2$ for
convenience.
By \eqref{E2:GGD}, 
\begin{equation}\label{EH:GGG}
 \begin{aligned}
\int \ud y\:& G_{\nu/2}(s,y-\bar{z})
\left(G_\nu(t-s,x-y)-G_\nu(t-s,x'-y)\right)^2
\\
 & \qquad =
 \frac{1}{\sqrt{4\pi\nu(t-s)}}
\left[G_{\nu/2}\left(t,x-\bar{z}\right)+G_{\nu/2}\left(t,x'-\bar{z}
\right)\right]\\
 & \qquad\qquad\qquad - 2\:G_{2\nu}\left(t-s,x-x'\right)
G_{\nu/2}\left(t,\bar{x}-\bar{z}\right).  
 \end{aligned}
\end{equation}
Then apply Lemma \ref{LH:IntGG} to integrate over $s$:
\begin{align*}
\int \ud s\: G\int \ud y\: G(G-G)^2 =&
\frac{1}{4\nu}\left(G_{\nu/2}\left(t,x-\bar{z}\right)+G_{\nu/2}\left(t,x'-\bar{z
} \right)\right)\Erfc\left(\frac{|z_1-z_2|}{\sqrt{4\nu t}}\right)\\
 &-\frac{1}{2\nu}
G_{\nu/2}\left(t,\bar{x}-\bar{z}\right)
\Erfc\left(
\frac{1}{\sqrt{2t}}
\left[\frac{|z_1-z_2|}{\sqrt{2\nu}}+\frac{\left|x-x'\right|}{\sqrt{2\nu}}\right]
\right)\;.
\end{align*}
It follows from the definition of $\Erfc(x)$ that $\Erfc\left(|x|+h\right)\ge
\Erfc\left(|x|\right) -\frac{2e^{-x^2}}{\sqrt{\pi}} h$ for $h\ge 0$ and we 
apply this inequality to the last factor to obtain,
\begin{align*}
\int \ud s \:G\int\ud y\: &G(G-G)^2 \\
\le &
\: \frac{1}{\nu}\:
G_{\nu/2}\left(t,\bar{x}-\bar{z}\right)
\frac{\left|x-x'\right|}{\sqrt{4\pi\nu t}} \exp\left(-\frac{(z_1-z_2)^2}{4\nu
t}\right)
\\
&+\frac{1}{4\nu}\Bigg[
G_{\nu/2}\left(t,x-\bar{z}\right)
+G_{\nu/2}\left(t,x'-\bar{z}\right)
-2
G_{\nu/2}\left(t,\bar{x}-\bar{z}\right)
\Bigg]
\Erfc\left(\frac{|z_1-z_2|}
{\sqrt{4\nu t}}\right). 
\end{align*}
Now apply Lemma \ref{LH:GG2G} with $h=\frac{x'-x}{2}$, $L=2n$ and
$\beta=1/2$: there are two constants
\begin{align*}
C_n'&=\sup_{s\in [1/n,n]} C_{2n,1/2,\nu s} = \frac{C \sqrt{n}}{\sqrt{2\nu}} +
\frac{1}{n},\: \quad C\approx 0.451256\;,\\
C_n'' &= \sup_{s\in [1/n,n]} C_{2n, 1/2, \nu s}'' = C_n'
\exp\left(\frac{6\:n^3}{\nu}\right)\;,
\end{align*}
where $C_{L,\beta,\nu s}'$ and $C_{L,\beta,\nu s}''$ are defined in
Lemma \ref{LH:GG2G},
such that for $\left|\frac{x-x'}{2}\right|\le\beta L = n$,
\begin{multline*}
\Bigg|G_{\nu/2}\left(t,x-\bar{z}\right)
+G_{\nu/2}\left(t,x'-\bar{z}\right)
-2
G_{\nu/2}\left(t,\bar{x}-\bar{z}\right)\Bigg| \\ 
\le \left\{
C_n''\:
\left[
G_{\nu/2}\left(t,\bar{x}-\bar{z}-2L\right)
+
G_{\nu/2}\left(t,\bar{x}-\bar{z}+2L\right)
\right]
+C_n'\:
G_{\nu/2}\left(t,\bar{x}-\bar{z}\right)
\right\}\left|x-x'\right|.
\end{multline*}
Note that $t\ge 1/n$ is essential for this inequality to be valid.
By Lemma \ref{LH:IntGG}, we have that $\Erfc\left(\frac{|z_1-z_2|}
{\sqrt{4\nu t}}\right) \le \sqrt{4\pi \nu t} \;
G_{2\nu}\left(t,z_1-z_2\right)$,
and so
\begin{align*}
\left|\int\ud s \: G\int\ud y\: G(G-G)^2\right|  \le
&
\left(\frac{1}{\nu} + \frac{\sqrt{\pi t}}{\sqrt{4\nu}}
C_n'\right)\left|x-x'\right|\:
G_{\nu/2}\left(t,\bar{x}-\bar{z}\right)
G_{2\nu}\left(t,z_1-z_2\right) \\
&+
\frac{\sqrt{\pi t} \; C_n'' }{\sqrt{4 \nu}}\left|x-x'\right|\:
G_{\nu/2}\left(t,\bar{x}-\bar{z}-2L\right)
G_{2\nu}\left(t,z_1-z_2\right)\\
&+
\frac{\sqrt{\pi t} \; C_n'' }{\sqrt{4 \nu}}\left|x-x'\right| \:
G_{\nu/2}\left(t,\bar{x}-\bar{z}+2L\right)
G_{2\nu}\left(t,z_1-z_2\right)\:.
\end{align*}
Now apply Lemma \ref{LH:Split}:
\begin{align*}
\left|\int \ud s\:G\int\ud y\: G(G-G)^2\right|  \le
&
\left(\frac{1}{\nu} + \frac{\sqrt{\pi n}}{\sqrt{4\nu}}
C_n'\right)\left|x-x'\right|
G_{2\nu}\left(t,\widetilde{x}_1-z_1\right)
G_{2\nu}\left(t,\widetilde{x}_1-z_2\right)\\
&+
\frac{\sqrt{\pi n} \; C_n'' }{\sqrt{4 \nu}}\left|x-x'\right|\:
G_{2\nu}\left(t,\widetilde{x}_2-z_1\right)
G_{2\nu}\left(t,\widetilde{x}_2-z_2\right)\\
&+
\frac{\sqrt{\pi n} \; C_n'' }{\sqrt{4 \nu}}\left|x-x'\right|\:
G_{2\nu}\left(t,\widetilde{x}_3-z_1\right)
G_{2\nu}\left(t,\widetilde{x}_3-z_2\right),
\end{align*}
where $\widetilde{x}_1 = \bar{x}$,  $\widetilde{x}_2 = \bar{x}-2L$
and $\widetilde{x}_3 = \bar{x}+2L$. Clearly, $\widetilde{x}_i\in
[-5n,5n]$ for all $i=1,2,3$.
Finally, after integrating over $|\mu|(\ud z_1)$ and $|\mu|(\ud z_2)$, we see
that
\[I(t,x;t,x') \le C_{n,3}' \left|x-x'\right|\] for all $t\in [1/n,n]$, and $x,
x'\in [-n,n]$, where the constant  is equal to
\[
C_{n,3}' =\left(\frac{1}{\nu} + \frac{\sqrt{\pi n}}{\sqrt{4\nu}}\left(
C_n'+2 C_n''\right)\right) \sup_{(s,y)\in[1/n,n]\times[-5n,5n]}
 \left(|\mu|*G_{2\nu}(s,\cdot)\right)^2\left(y\right).
\]
As for the contribution of the constant $\Vip$, it corresponds to the initial
data $|\mu|(\ud x) \equiv \Vip \ud x$ and we apply 
Proposition \ref{PH:G}.
Finally, one can choose, for \eqref{E2:H3}, 
\begin{align*}
C_{n,3} =
\frac{\Vip^2}{\nu} 
+
\left(\frac{2}{\nu} + \frac{\sqrt{\pi n}}{\sqrt{\nu}}\left(
C_n'+2 C_n''\right)\right) \sup_{(s,y)\in[1/n,n]\times [-5n,5n]}
  \left(|\mu|*G_{2\nu}(s,\cdot)\right)^2\left(y\right).
\end{align*}
This constant $C_{n,3}$ is clearly finite. This
completes the proof of \eqref{E2:H3}.

{\vspace{1em}\noindent\bf Property \eqref{E2:H5}.~~}
We first consider the contribution of $J_0^*(t,x)$. As before, let
\[I\left(t,x;t',x'\right)
=\iint_{[t,t']\times\R} \ud s\ud y\: \left[J_0^*\left(s,y\right)\right]^2
G_\nu^2(t'-s,x'-y).
\]
Set $\bar{z}=(z_1+z_2)/2$.
Similar to the arguments leading to \eqref{E2:Itxtx}, we have
\begin{equation}
 \label{E2:Itxtx5}
\begin{aligned}
 I\left(t,x;t',x'\right)
=&\int_t^{t'} \ud s \iint_{\R^2}
|\mu|(\ud z_1)|\mu|(\ud z_2)\:
G_{2\nu}(s,z_1-z_2)
\\
&\times \int_\R\ud y\: G_{\nu/2}\left(s,y-\bar{z}\right)
G_\nu^2\left(t'-s,x'-y\right).
\end{aligned}
\end{equation}
Apply Lemma \ref{LH:GG} to $G_\nu^2\left(t'-s,x'-y\right)$ and then
integrate over $y$,
\[
I\left(t,x;t',x'\right)=
\int_t^{t'}\ud s\iint_{\R^2}
|\mu|(\ud z_1)|\mu|(\ud z_2) \frac{1}{\sqrt{4\pi \nu (t'-s)}}
G_{2\nu}(s,z_1-z_2)
G_{\nu/2}\left(t',x'-\bar{z}\right).
\]
Now apply Lemma \ref{LH:Split} to $G_{2\nu}(s,z_1-z_2)
G_{\nu/2}\left(t',x'-\bar{z}\right)$.  Then by Lemma \ref{L2:arcsin}
and the fact that
$\arcsin(x) \le \pi x /2$ for $x\in [0,1]$,
we see that
\[I\left(t,x;t',x'\right)  \le
\left|J_0^*\left(2t',x'\right)
\right|^2 \frac{2 \sqrt{t'}}{\sqrt{\pi
\nu }}\arcsin\left(\sqrt{\frac{t'-t}{t'}}\right)
\le
\left|J_0^*\left(2t',x'\right)
\right|^2 \sqrt{\frac{\pi}{\nu}} \sqrt{t'-t}.\]
Therefore,
\[
I\left(t,x;t',x'\right) \le C_{n,5}' \sqrt{t'-t},
\quad\text{with $C_{n,5}' =  \sqrt{\frac{\pi}{\nu}}
\sup_{(s,y)\in [1/n,n]\times[-n,n]}
\left|J_0^*(2s,y)\right|^2$.}
\]
As for the contribution of $\Vip$, it corresponds to the initial data
$|\mu|(\ud x)\equiv \Vip\ud x$ and we apply Proposition \ref{PH:G}. Finally, we 
can choose
\begin{align}\label{E2:Cn5}
C_{n,5} = \frac{\Vip^2 }{\sqrt{\pi \nu}}+
2\sqrt{\frac{\pi}{\nu}} \sup_{(s,y)\in [1/n,n]\times[-n,n]} 
\left|J_0^*(2s,y)\right|^2
\end{align}
for \eqref{E2:H5}. This completes the proof of \eqref{E2:H5} and therefore part 
(1) of Proposition \ref{P2:Holder-I}.
\end{proof}

\vspace{1em}
\begin{proof}[Proof of Proposition \ref{P2:Holder-II} (1)]
We first prove \eqref{E2:H2} and \eqref{E2:H4}. Denote
\[
I(t,x;t',x')= \iint_{[0,t]\times\R}\ud s\ud y
\left(\left|J_0^*\right|^2\star
G_\nu^2\right)\left(s,y\right)
\left(
G_\nu\left(t-s,x-y\right)-G_\nu(t'-s,x'-y)
\right)^2.
\]
Let $\bar{z}=(z_1+z_2)/2$. As in \eqref{E2:Itxtx}, replace
$\left|J_0^*(u,z)\right|^2$ by the double integral.
By Lemma \ref{LH:GG},
\begin{align*}
I\left(t,x;t',x'\right)=& \iint_{\R^2}|\mu|(\ud z_1)|\mu|(\ud z_2)
\int_0^t  \ud s \int_0^s  \ud u \: \frac{1}{\sqrt{4\nu\pi (s-u)}}
G_{2\nu}(u,z_1-z_2) \\
&\times \iint_{\R^2}\ud y \ud z\:
G_{\nu/2}\left(u,z-\bar{z}\right)  G_{\nu/2}(s-u,y-z)\\
&\times\left(G_\nu\left(t-s,x-y\right)-G_\nu\left(t'-s,x'-y\right)\right)^2.
\end{align*}
We first integrate over $\ud z$ using the semigroup property and then
integrate over $\ud u$ by using Lemma \ref{LH:IntGG} and use the fact that
$s\le t\le n$ to obtain
\begin{equation}\label{E2:Itxtx2}
\begin{aligned}
 I\left(t,x;t',x'\right)\le &\frac{\sqrt{\pi n}}{\sqrt{4\nu}}
\int_0^t  \ud s
\iint_{\R^2}|\mu|(\ud z_1)|\mu|(\ud z_2)\: G_{2\nu}(s,z_1-z_2) \\
&\times \int_\R\ud y\:
G_{\nu/2}\left(s,y-\bar{z}\right)
\left(G_\nu\left(t-s,x-y\right)-G_\nu\left(t'-s,x'-y\right)\right)^2 \;.
\end{aligned}
\end{equation}
Comparing this upper bound with \eqref{E2:Itxtx}, we can apply Proposition
\ref{P2:Holder-I} to conclude that \eqref{E2:H2} and \eqref{E2:H4} are
true with the constants $C_{n,2}$ and $C_{n,4}$ given in \eqref{E2:Cn246}.
As for \eqref{E2:H6}, let
\[
I\left(t,x;t',x'\right) =
\iint_{[t,t']\times\R} \ud s \ud y \left(\left|J_0^*\right|^2\star
G_\nu^2 \right) \left(s,y\right)
G_\nu^2(t'-s,x'-y).
\]
By arguments similar to those leading to \eqref{E2:Itxtx2}, we have that
\begin{align*}
I\left(t,x;t',x'\right)\le &\frac{\sqrt{\pi n}}{4\nu}
\int_t^{t'} \ud s
\iint_{\R^2}  |\mu|(\ud z_1)|\mu|(\ud z_2)G_{2\nu}(s,z_1-z_2) \\
& \quad\qquad \times \int_\R\ud y\: 
G_{\nu/2}\left(s,y-\bar{z}\right) G_{\nu}^2(t'-s,x'-y).
\end{align*}
Comparing this upper bound with \eqref{E2:Itxtx5}, we can apply Proposition
\ref{P2:Holder-I} to conclude that \eqref{E2:H6} is true with the
corresponding constant $C_{n,6}$ given in \eqref{E2:Cn246}.
This completes the proof of part (1) of Proposition \ref{P2:Holder-II}.
\end{proof}

\subsection{Proofs of part (2) of the Propositions \ref{P2:Holder-I} and 
\ref{P2:Holder-II}} \label{sec4.4}

\begin{lemma}\label{LH:linExp}
For $a\ge 1$ and $b\ge (a\:e)^{-1}$, we have that $|x|\le e^{b|x|^a}$ for all 
$x\in\R$.
\end{lemma}
\begin{proof}
The case where $x=0$ is clearly true. We only need to consider the case where 
$x>0$. 
Equivalently, we need to solve the critical case where the graphs of the two 
functions $\log x$ and $b \: x^a$ intersect exactly once ($x>0$), that is,  
\[
\log x = b \: x^a,\quad\text{and}\quad \frac{1}{x} = a\: b\: x^{a-1},
\]
which implies $x=e^{1/a}$ and $b=(a\:e)^{-1}$. When $b$ is bigger than 
this critical value, 
the function $b|x|^a$ will dominate $\log x$ for all $x>0$.
\end{proof}

\begin{lemma}\label{LH:Delta-g}
Let $g(x)=e^{c |x|^a}$ with $c>0$ and $a>1$. For all $n>0$, the following 
properties hold:\\
(1) For all $x, z\in\R$, $0\le t\le t'\le n$, 
\[
\left|
g\left(x-\sqrt{t} \:z\right)
-g\left(x-\sqrt{t'} \:z\right)
\right|
\le 
a\:c\:\exp\left(c_1 |x|^a + c_2 |z|^a\right) \left|t'-t\right|^{1/2},
\]
where the two constants $c_1:=c_1(a,c)$ and $c_2:=c_2(n,a,c)$ can be chosen as 
follows:
\[
c_1(a,c)=\left(c+\frac{a-1}{a\: e}\right) 2^{a-1},\quad\text{and}\quad 
c_2(n,a,c)= c_1(a,c)\: n^{a/2} + \frac{1}{a\: e}.
\]\\
(2) For all $x,x'\in [-n,n]$, $z\in\R$ and $t\in [0,n]$,
\[
\left| 
g\left(x-\sqrt{t}\: z\right)
-g\left(x'-\sqrt{t}\: z\right)
\right|
\le c_3 \:\exp\left(c_4 |z|^a\right)\: \left|x'-x\right|
\]
where the two constants $c_3:=c_3(n,a,c)$ and $c_4:=c_4(n,a,c)$ can be chosen 
as 
follows:
\[
c_3(n,a,c): = a\: c\: e^{c_1\: n^a},\quad\text{and}\quad
c_4(n,a,c)= c_1\: n^{a/2}.
\]
\end{lemma}
\begin{proof}
(1) Because $a>1$, the function $g$ belongs to $C^1(\R)$, is convex and 
$g'(x)\ge 0$ for $x\ge 0$.
Hence, 
\[
\left|
g\left(x-\sqrt{t} \:z\right)
-g\left(x-\sqrt{t'} \:z\right)
\right|
\le \left|g'\left(|x|+\sqrt{n}\: |z|\right)\right|\cdot \left|\sqrt{t'}\: 
z-\sqrt{t}\: z\right|.
\]
Let $b=(a\: e)^{-1}$.
By Lemma \ref{LH:linExp}, 
$|g'(x)|=a\:c\:|x|^{a-1} e^{c |x|^a}\le a\:c\: e^{(c+(a-1)b)|x|^a}$. Thus
\begin{align}\label{EH_:Dg}
 \left|g'\left(|x|+\sqrt{n}\: |z|\right)\right|
\le a\:c\:e^{(c+(a-1)b)\left(|x|+\sqrt{n}\: |z|\right)^a}
\le 
a\:c\:e^{c_1 |x|^a + c_1 n^{a/2} |z|^a},
\end{align}
where we have applied the inequality $(x+y)^a\le 2^{a-1} (x^a+y^a)$ for all 
$x,y\ge 0$.
Clearly, 
\begin{align}
\label{EH:Delta-t}
|\sqrt{t'}-\sqrt{t}|\le \sqrt{t'-t}.
\end{align}
Finally, apply Lemma \ref{LH:linExp} to $|z|$, and combining all the above 
bounds proves (1).

(2) Similarly to (1), 
\[
\left|
g\left(x-\sqrt{t} \:z\right)
-g\left(x'-\sqrt{t} \:z\right)
\right|
\le \left|g'\left(|n|+\sqrt{n}\: |z|\right)\right|\cdot \left|x-x'\right|,
\]
and by \eqref{EH_:Dg}, $\left|g'\left(|n|+\sqrt{n}\: |z|\right)\right|
\le 
a\: c\: e^{c_1 n^a + c_1 n^{a/2} |z|^a}$. This proves (2).
\end{proof}

For $c>0$ and $a\in [0,2[\:$, define the constant
\[
K_{a,c}(\nu t):=\left(e^{c|\cdot|^a}* G_{\nu}(t,\cdot)\right)(0)\:.
\]
For $0\le t\le n$, we have that
\begin{align}\label{EH:Kact}
 K_{a,c}(\nu t)=\int_\R \ud y\: e^{c \left(\sqrt{t} \:|y|\right)^a} G_{\nu}(1,y)
\le 
\int_\R \ud y \: e^{c \left(\sqrt{n} \:|y|\right)^a} G_{\nu}(1,y)
= K_{a,c}(\nu n).
\end{align}

\vspace{1em}
\begin{proof}[Proof of Proposition \ref{P2:Holder-I} (2)]
Because $\mu\in\calM_H^*(\R)$, there are a function $f(x)$ and two constants 
$a\in [1,2[\:$ and $c>0$ such 
that $\mu(\ud x)=f(x)\ud x$ and $c=\sup_{x\in\R} |f(x)|e^{-|x|^a}<+\infty$.
In the following, we assume that $x,x'\in [-n,n]$, and $t,t'\in [0,n]$.
Set $g(x)=e^{2^a |x|^a}$ and assume that $\Vip=0$. 
From \eqref{E2:Itxtx}, 
\begin{align*}
 I(t,x;t',x')\le& c^2\int_0^t\ud s \iint_{\R^2} \ud z_1 \ud z_2 \:e^{|z_1|^a + 
|z_2|^a} G_{2\nu}(s,z_1-z_2)
\\
&\times \int_\R  \ud y\: G_{\nu/2}\left(s,y-\bar{z}\right)
\left(
G_\nu\left(t-s,x-y\right)-G_\nu(t'-s,x'-y)
\right)^2.
\end{align*}
We shall apply the change of variables $z=\bar{z}$ and $w=\Delta z$: since 
\[
 |z_1|^a+|z_2|^a
= \left|z+\frac{w}{2}\right|^a + \left|z-\frac{w}{2}\right|^a
\le 2^{a-1}\left(|z|^a + \left|\frac{w}{2}\right|^a\right) \times 2
= 2^a|z|^a + |w|^a,
\]
we see that 
\begin{align*}
e^{|z_1|^a + |z_2|^a}\le e^{2^a|z|^a + |w|^a} = e^{|w|^a} g(z),
\end{align*}
and it follows that
\begin{align}
\notag
I(t,x;t',x')\le&\:
 c^2 \int_0^t \ud s \int_\R \ud z \left(e^{|\cdot|^a}* 
G_{2\nu}(s,\cdot)\right)(0)  g(z)
\\
\notag &\qquad\times \int_\R\ud y
\:  G_{\nu/2}(s,y-z)\left(G_\nu(t-s,x-y)-G_\nu(t'-s,x'-y)\right)^2\\
\notag \le& 
 \:c^2 \:K_{a,1}(2\nu n) \int_0^t \ud s \int_\R \ud z \: g(z) \\
& \qquad \times \int_\R\ud y
\:  G_{\nu/2}(s,y-z)\left(G_\nu(t-s,x-y)-G_\nu(t'-s,x'-y)\right)^2,
\label{EH:txtx}
\end{align}
where the second inequality is due to \eqref{EH:Kact}.

{\vspace{1em}\bf \noindent Property \eqref{E2:H1}.~~}
For the moment, we continue to assume that $\Vip=0$.  Set $x=x'$.
Apply \eqref{E2:GGD} with $x=x'$ and $\bar{z}$ replaced by $z$, integrate 
over $\ud z$, and use \eqref{EH:txtx} to see that,
\begin{align*}
I(t,x;t',x)
&\le 
c^2\: K_{a,1}(2 \nu n)\int_0^t \ud s \Bigg(
\frac{1}{\sqrt{4\pi\nu(t-s)}} \left(g*G_{\nu/2}(t,\cdot)\right)(x)
\\
&\qquad+
\frac{1}{\sqrt{4\pi\nu(t'-s)}} \left(g*G_{\nu/2}(t',\cdot)\right)(x)\\
&\qquad-
\frac{2}{\sqrt{4\pi\nu\left(\frac{t+t'}{2}-s\right)}} 
\left(g*G_{\nu/2}\left(t+\frac{(t-s)h}{2(t-s)+h},\cdot\right)\right)(x)
\Bigg)\\
&\le c^2 \: K_{a,1}(2\nu n) \int_0^t \ud s \left(I_1+I_2+I_3\right),
\end{align*}
where, letting $h=t'-t$,  
\begin{align*}
I_1 &= 
\left[\left(4\pi\nu (t-s)\right)^{-\frac{1}{2}}+\left(4\pi\nu
(t'-s)\right)^{-\frac{1}{2}}-
\left(\pi \nu\left(\frac{t+t'}{2}-s\right)\right)^{-\frac{1}{2}}\right]
\left(g*G_{\nu/2}(t,\cdot)\right)(x),\\
I_2 &= \frac{1}{\sqrt{4\pi\nu (t'-s)}} \left[
\left(g*G_{\nu/2}(t',\cdot)\right)(x)-\left(g*G_{\nu/2}(t,\cdot)\right)(x)\right
], \\
I_3 &= 
\frac{2}{\sqrt{4\pi\nu \left(\frac{t+t'}{2}-s\right)}} \left[
\left(g*G_{\nu/2}\left(t+\frac{(t-s)h}{2(t-s)+h},
\cdot\right)\right)(x)-\left(g*G_{\nu/2}(t,\cdot)\right)(x)\right].
\end{align*}
Set $\bar{t} = \frac{t+t'}{2}$. By \eqref{EH:Delta-t},
\begin{align*}
\int_0^t I_1 \ud s
&= \frac{1}{\sqrt{\pi\nu}}\left(
\sqrt{t} +\sqrt{t'} - \sqrt{h} -2 \sqrt{\bar{t}} + 2\sqrt{\bar{t}-t}
\right) \: \left(g*G_{\nu/2}(t,\cdot)\right)(x) \\
&\le
\frac{1}{\sqrt{\pi\nu}}\left(
\left|\sqrt{t}-\sqrt{\bar{t}}\right| +\left|\sqrt{t'} - \sqrt{\bar{t}}\right| 
-\sqrt{h} + 2\sqrt{\frac{h}{2}}
\right)\: \left(g*G_{\nu/2}(t,\cdot)\right)(x)\\
&\le
\frac{1}{\sqrt{\pi\nu}}\left(
4\sqrt{\frac{h}{2}} - \sqrt{h}
\right) \: \left(g*G_{\nu/2}(t,\cdot)\right)(x) .
\end{align*}
By Lemma \ref{LH:Delta-g}, for some constants $c_i>0$, $i=1,2$,
\begin{align*}
|I_2| &\le \frac{1}{\sqrt{4\pi\nu (t'-s)}}
\int_\R \ud z\; 
\left| g\left(x-\sqrt{t}\: z\right)-g\left(x-\sqrt{t'}\: z\right)\right| 
G_{\nu/2}(1,z) \\
&\le 
\frac{1}{\sqrt{4\pi\nu (t-s)}} a\: 2^a e^{c_1|x|^a} 
\left(e^{c_2|\cdot|^a}*G_{\nu/2}(1,\cdot)\right)(0)\: \sqrt{h}\;.
\end{align*}
Hence, for all $0\le t\le t'\le n$ and $x\in [-n,n]$,
\[
\int_0^t \ud s \: |I_2| \le 
\frac{a \: 2^a \sqrt{n}}{\sqrt{\pi\nu}}  e^{c_1|n|^a} 
K_{a,c_2}\left(\frac{\nu}{2}\right)\: 
\sqrt{h}\:.
\]
Similarly, because $\frac{(t-s)h}{2(t-s)+h}\le \frac{h}{2}$,
for all $0\le s\le  t\le t'\le n$ and $x\in [-n,n]$,
\[
\int_0^t \ud s\: |I_3| \le 
\frac{a \: 2^a \sqrt{n}}{\sqrt{2\pi\nu}}  e^{c_1|n|^a} 
K_{a,c_2}\left(\frac{\nu}{2}\right)\: 
\sqrt{h}\:.
\]
Therefore, for all $0\le t\le t'\le n$ and $x\in [-n,n]$,
$I(t,x;t',x) 
\le \widetilde{C}_{n,1}^* \sqrt{t'-t}$ with 
\begin{multline*}
\widetilde{C}_{n,1}^* =
\frac{c^2\:K_{a,1}(2 \nu n)}{\sqrt{2\pi\nu}}\Bigg[\left(\sqrt{2}+1\right)\: a 
\: 2^a \sqrt{n} \: e^{c_1|n|^a} K_{a,c_2}\left(\frac{\nu}{2}\right)\:
\\+\left(4-\sqrt{2}\right) \sup_{(s,y)\in [0,n]\times 
[-n,n]}\left(e^{2^a|\cdot|^a}*G_{\nu/2}(s,\cdot)\right)(y)\Bigg].
\end{multline*}
Finally, as for \eqref{E2:H1}, the contribution of the constant $\Vip$ can be 
calculated by using
Proposition \ref{PH:G}. Therefore, one can choose 
\[
C_{n,1}^* = \Vip^2 \frac{\sqrt{2}-1}{\sqrt{\pi\nu}} + 2 \:\widetilde{C}_{n,1}^*.
\]

{\bf\noindent Property \eqref{E2:H3}.~~}
Assume again that $\Vip=0$. 
Set $t=t'$ and $\bar{x}=\frac{x+x'}{2}$. Recalling \eqref{EH:GGG}, we see 
that the inequality \eqref{EH:txtx} reduces to
\begin{multline*}
 I(t,x;t,x') \le  c^2\: K_{a,1}(2\nu n) \int_0^t \ud s \int_\R \ud z 
\: g(z)\Bigg\{ \frac{1}{\sqrt{4\pi\nu (t-s)}} 
\left[G_{\nu/2}(t,x-z)+G_{\nu/2}(t,x'-z)\right]\\
- 2\: G_{2\nu}(t-s,x-x') \: G_{\nu/2}(t,\bar{x}-z)
\Bigg\}.
\end{multline*}
Then integrate over $\ud s$ using Lemma \ref{LH:IntGds}: 
\begin{multline*}
 I(t,x;t,x') \le  c^2\: K_{a,1}(2\nu n) \int_\R \ud z 
\: g(z)\Bigg\{ \frac{\sqrt{t}}{\sqrt{\pi\nu}} 
\left[G_{\nu/2}(t,x-z)+G_{\nu/2}(t,x'-z)\right]\\
-2 \left[ 2 t \: G_{2\nu}(t,x-x') - \frac{1}{2\nu} \: |x-x'| \: 
\Erfc\left(\frac{|x-x'|}{\sqrt{4\nu t}}\right)\right] G_{\nu/2}(t,\bar{x}-z)
\Bigg\}.
\end{multline*}
Denote $F(x)=\left(g* G_{\nu/2}(t,\cdot)\right)(x)$. Then integrating over $\ud 
z$ gives
\begin{align*}
 I(t,x;t,x') \le&\:  c^2\: K_{a,1}(2\nu n) \Bigg\{ 
\frac{\sqrt{t}}{\sqrt{\pi\nu}} 
\left[F(x)+F(x')\right]\\
& \qquad\quad - 2\left[ \frac{\sqrt{t}}{\sqrt{\pi\nu}} 
e^{-\frac{(x-x')^2}{4\nu t}} - \frac{1}{2\nu} \: |x-x'| \: 
\Erfc\left(\frac{|x-x'|}{\sqrt{4\nu t}}\right)\right]  F(\bar{x})
\Bigg\}\\
\le &\:
c^2\: K_{a,1}(2\nu n) \Bigg\{ 
 \frac{\sqrt{t}}{\sqrt{\pi\nu}} \left|F(x)-F(\bar{x})\right|
+ \frac{\sqrt{t}}{\sqrt{\pi\nu}} \left|F(x')-F(\bar{x})\right|\\
& + \frac{2\sqrt{t}}{\sqrt{\pi\nu}} \left(1- e^{-\frac{|x-x'|^2}{4\nu 
t}}\right) F(\bar{x})+ \frac{1}{\nu} \: |x-x'|\: F(\bar{x})\Bigg\}.
\end{align*}
Notice that $0\le 1-e^{-x^2/2} \le \widetilde{C} \: |x|$, where the universal 
constant $\widetilde{C}$ is given in Lemma \ref{LH:GG2G}. 
By part (2) of Lemma \ref{LH:Delta-g}, for some constants $c_i$, $i=3,4$, 
\begin{align*}
\left|F(x)-F(\bar{x})\right|
 &\le \int_\R\ud z 
 \left|g\left(x-\sqrt{t}\: z\right)-g\left(\bar{x}-\sqrt{t}\: z\right)\right| 
G_{\nu/2}(1,z)\\
 &\le c_3 \left(e^{c_4 |\cdot|^a} * G_{\nu/2}(1,\cdot)\right)(0) \: \left|x- 
\bar{x}\right|\\
 &= \frac{c_3}{2} \: K_{a,c_4}\left(\frac{\nu}{2}\right)\: |x-x'|.
\end{align*}
Similarly, $\left|F(x')-F(\bar{x})\right|\le \frac{c_3}{2} \: 
K_{a,c_4}\left(\frac{\nu}{2}\right)\: |x-x'|$. Hence, 
\begin{align*}
 I(t,x;t,x') \le c^2\: K_{a,1}(2\nu n) \Bigg\{ 
 \frac{c_3 \: \sqrt{n}}{\sqrt{\pi\nu}} K_{a,c_4}\left(\frac{\nu}{2}\right)
 + \left(\frac{\widetilde{C}\:\sqrt{2}}{\nu\sqrt{\pi}} + 
\frac{1}{\nu}\right)  
F(\bar{x})\Bigg\}\:|x-x'|\:.
\end{align*}
Therefore, for all $0\le t\le n$ and $x,x'\in [-n,n]$, $I(t,x;t,x')
\le \widetilde{C}_{n,3}^* \left|x-x'\right|$ with 
\[
\widetilde{C}^*_{n,3}= c^2\: K_{a,1}(2\nu n) \Bigg\{ 
 \frac{c_3 \: \sqrt{n}}{\sqrt{\pi\nu}} K_{a,c_4}\left(\frac{\nu}{2}\right)
 + \left(\frac{\widetilde{C}\:\sqrt{2}}{\nu\sqrt{\pi}} + 
\frac{1}{\nu} \right)
\sup_{(s,y)\in [0,n]\times[-n,n]}
\left(g*G_{\nu/2}(s,\cdot)\right)(y)
\Bigg\}\;,
\]
and $\widetilde{C}_{n,3}^*<+\infty$ by definition of $g$.
Finally, the contribution of the constant $\Vip$ in \eqref{E2:H3} is given in
Proposition \ref{PH:G}. Therefore, one can choose 
\[
C_{n,3}^* = \frac{\Vip^2}{\nu} + 2 \: \widetilde{C}_{n,3}^*\:.
\]

{\bf\noindent Property \eqref{E2:H5}.~~}
As for \eqref{E2:H5}, notice that $J_0^*(t,x)\le c \left(e^{|\cdot|^a} * 
G_\nu(t,\cdot)\right)(x)$.
By checking the proof of part (1) (see \eqref{E2:Cn5}), one can choose, 
\begin{align*}
C_{n,5}^* = \frac{\Vip^2 }{\sqrt{\pi \nu}}+
2\:c^2 \sqrt{\pi/\nu} \sup_{(s,y)\in [0,n]\times[-n,n]} \left(e^{|\cdot|^a} * 
G_\nu(2s,\cdot)\right)^2(y).
\end{align*}
This completes the proof of part (2) of Proposition \ref{P2:Holder-I}.
\end{proof}

\vspace{1em}
\begin{proof}[Proof of Proposition \ref{P2:Holder-II} (2)]
If $\mu\in\calM_H^*(\R)$, then by Proposition \ref{P2:Holder-I} (2), the 
l.h.s. of \eqref{E2:H2} is
bounded by
\[C_{n,1}^*\sqrt{t'-t}\left(1 \star G_\nu^2 \right)(t,x)
=C_{n,1}^* \frac{\sqrt{t}}{\sqrt{\pi \nu}} \sqrt{t'-t}
\le C_{n,1}^* \frac{\sqrt{n}}{\sqrt{\pi \nu}} \sqrt{t'-t}.
\]
Hence, $C_{n,2}^* = \frac{\sqrt{n}}{\sqrt{\pi\nu}} C_{n,1}^*$. 
The same arguments apply to the other two constants $C_{n,4}^*$ and 
$C_{n,6}^*$, i.e., \eqref{E2:H4} and \eqref{E2:H6}.
Note that it was not possible to use the above argument in the proof of part 
(1) of Proposition \ref{P2:Holder-I}. This completes the proof of Proposition 
\ref{P2:Holder-II} (2).
\end{proof}

\subsection{Checking the initial condition}
\label{ss:Weak}
\begin{proof}[Proof of Proposition \ref{PH:Weak}]
Because $u(t,x) = J_0(t,x)+I(t,x)$, and because it is standard that (see 
\cite[Chapter 7, Section 6]{friedman} and also \cite[Lemma 2.6.14, 
p.89]{LeChen13Thesis}),
\[
\lim_{t\rightarrow 0_+} \int_\R \ud x \: J_0(t,x) \phi(x) = \int_\R\mu(\ud x)\: 
\phi(x),
\]
we only need to prove that 
\[
\lim_{t\rightarrow 0_+} \int_\R \ud x \: I(t,x) \phi(x) = 0 \quad\text{in 
$L^2(\Omega)$}.
\]
Recall that the Lipschitz continuity of $\rho$ implies the linear growth 
condition \eqref{EH:LinGrow}. Fix $\phi\in C_c^\infty(\R)$.
Denote $L(t):=\int_\R I(t,x) \phi(x)\ud x$.
By the stochastic Fubini theorem (see \cite[Theorem 2.6, p. 296]{Walsh86}), 
whose assumptions are easily checked,
\[
L(t) =  \int_0^t \int_\R \left(\int_\R \ud x\;  G_\nu(t-s,x-y) \phi(x)\right) 
\rho(u(s,y)) W(\ud s,\ud y).
\]
Hence, by \eqref{EH:LinGrow},
\begin{align*}
\E\left[L(t)^2\right]\le 
\Lip_\rho^2 \int_0^t \ud s \int_\R  \ud y 
\left(\int_{\R} \ud x\; G_\nu(t-s,x-y)\phi(x)\right)^2
\left(\Vip^2+\Norm{u(s,y)}_2^2 \right).
\end{align*}
By the moment formula \eqref{E2:SecMom-Up}, we can write the above upper bound 
as
\[
\E\left[L(t)^2\right]\le \Lip_\rho^2 \left[L_1(t) + 
L_2(t)+L_3(t)+L_4(t)\right],
\] 
where
\begin{align*}
L_1(t) &=
\int_0^t \ud s \int_\R  \ud y 
\left(\int_{\R} \ud x\; G_\nu(t-s,x-y)\phi(x)\right)^2
J_0^2(s,y),\\ 
L_2(t) &=
\int_0^t \ud s \int_\R  \ud y 
\left(\int_{\R} \ud x\; G_\nu(t-s,x-y)\phi(x)\right)^2
\left(J_0^2\star \overline{\calK}\:\right)(s,y),\\
L_3(t) &=
\Vip^2 \int_0^t \ud s \: \overline{\calH}(s) \int_\R  \ud y 
\left(\int_{\R} \ud x\; G_\nu(t-s,x-y)\phi(x)\right)^2,\\
L_4(t) &=
\Vip^2 \int_0^t \ud s  \int_\R  \ud y 
\left(\int_{\R} \ud x\; G_\nu(t-s,x-y)\phi(x)\right)^2.
\end{align*}
From now on, we may assume that $\mu\in \calM_{H,+}(\R)$, because otherwise, 
one 
can simply replace the above $J_0(s,y)$ by $J_0^*(s,y) = \left(|\mu|* 
G_\nu(s,\circ)\right)(y)$.

{\vspace{1em}\bf\noindent (1)~~}
Consider $L_1(t)$ first. Write out both $J_0^2(s,y)$ and $\left(\int_\R 
\ud x \; G_\nu(t-s,x-y) \phi(x)\right)^2$ in the forms of double integrals, and 
apply 
Lemma \ref{LH:GG}, to see that
\begin{equation}
 \label{EH:L1}
\begin{aligned}
L_1(t) =&
\int_0^t \ud s \int_\R  \ud y \left(\iint_{\R^2} \ud x_1 \ud 
x_2 \; G_{\nu/2}(t-s,\bar{x}-y)G_{2\nu}(t-s,\Delta x) 
\phi(x_1)\phi(x_2) \right)\\
 &\qquad\times\iint_{\R^2} \mu(\ud z_1)\mu(\ud z_2) \; 
G_{\nu/2}(s,\bar{z}-y)G_{2\nu}(s,\Delta z), 
\end{aligned}
\end{equation}
where $\bar{x}= \frac{x_1+x_2}{2}$, $\Delta x = x_1-x_2$ and similarly for 
$\bar{z}$ and $\Delta z$.
Integrate over $\ud y$ first using the semigroup property of the heat 
kernel and then integrate over $\ud s$ by using Lemma \ref{LH:IntGG}, we see 
that
\[
L_1(t)= \iint_{\R^2}\ud x_1\ud x_2 \; \phi(x_1)\phi(x_2) 
\iint_{\R^2}\mu(\ud z_1)\mu(\ud z_2)\; G_{\nu/2}(t,\bar{x}-\bar{z}) 
\frac{1}{4\nu} \Erfc\left(\frac{1}{\sqrt{4\nu t}}\left[|\Delta x|+ 
|\Delta z|\right]\right).
\]
By \eqref{EH:ErfcBd}, 
\begin{align*}
\Erfc\left(\frac{1}{\sqrt{4\nu t}}\left[|\Delta x|+ 
|\Delta z|\right]\right)
&\le e^{-\frac{\left(|\Delta x|+|\Delta z|\right)^2}{4\nu t}}
\le e^{-\frac{|\Delta x|^2}{4\nu t}}
e^{-\frac{|\Delta z|^2}{4\nu t}}=4\pi \nu \sqrt{t}\: 
G_{2\nu}\left(1,\frac{\Delta x}{\sqrt{t}}\right) 
G_{2\nu}(t,\Delta z).
\end{align*}
By the change of variables $y=(x_1+x_2)/2$ and $w=(x_1-x_2)/\sqrt{t}$,
\begin{align*}
L_1(t)\le&\pi\: t  \iint_{\R^2} \ud y \ud w\:G_{2\nu}(1,w)
\phi\left(y+\frac{\sqrt{t}}{2}w\right)\phi\left(y-\frac{\sqrt{t}}{2}w\right)\\
&\qquad\quad\times \iint_{\R^2}\mu(\ud z_1)\mu(\ud z_2) \: 
G_{\nu/2}(t,y-\bar{z}) 
G_{2\nu}(t,\Delta z) .
\end{align*}
By Lemma \ref{LH:Split}, 
\[
\iint_{\R^2}\mu(\ud z_1)\mu(\ud z_2) \: G_{\nu/2}(t,y-\bar{z}) 
G_{2\nu}(t,\Delta z) \le 2\left(\mu * G_{2\nu}(t,\cdot)\right)^2(y) = 2 
J_0^2(2t,y).
\]
For some constants $a$ and $c\ge 0$, $|\phi(x)|\le c \: 1_{[-a,a]}(x)$. 
If $|y|>a$, then the two sets $\left\{w\in\R: 
\left|\frac{\sqrt{t}}{2}w\pm y\right|\le a\right\}$ have empty intersection.
Hence, 
\begin{align*}
L_1(t) &\le 2 c^2 \pi  \int_{|y|\le a}\ud y\: t \:J_0^2(2t ,y) \int_\R\ud w\: 
G_{2\nu}(1,w )
=2 c^2 \pi  \int_{|y|\le a}\ud y\: t \:J_0^2(2t ,y).
\end{align*}
Clearly, by assuming that $t\le 1$,
\[
\sqrt{t}\: J_0(2t,y) =
\int_{\R} \mu(\ud x) \frac{1}{\sqrt{4\pi \nu}} 
e^{-\frac{(y-x)^2}{4\nu t}} 
\le 
\int_{\R} \mu(\ud x) \frac{1}{\sqrt{4\pi \nu}} 
e^{-\frac{(y-x)^2}{4\nu}} = 
J_0(2,y).
\]
Hence, Lebesgue's dominated convergence theorem implies that
\[
\lim_{t\rightarrow 0} \sqrt{t} \: J_0(2t, y) = 0.
\]
Because $\int_{|y|\le a} \ud y\: J_0^2(2,y)<+\infty$, by another application of 
Lebesgue's dominated convergence theorem, we see that $\lim_{t\rightarrow 0} 
L_1(t) =0$.

{\vspace{1em}\bf\noindent (2)~~}
As for $L_2(t)$, because $\overline{\calK}(t,x)\le G_{\nu/2}(t,x) 
\frac{1}{\sqrt{t}}
h(t)$, where $h(t) := 
\Lip_\rho^2(\:\frac{1}{\sqrt{4\pi\nu}}+\frac{\Lip_\rho^2 
\sqrt{t}}{2\nu}e^{\frac{\Lip_\rho^4 t}{4\nu}}\:)$
is a nondecreasing function in $t$, we see that as in \eqref{EH:L1},
\begin{align*}
 L_2(t)\le& 
\int_0^t \ud s \int_\R  \ud y 
\iint_{\R} \ud x_1\ud x_2\; 
G_{\nu/2}(t-s,\bar{x}-y)G_{2\nu}(t-s,\Delta x)\phi(x_1)\phi(x_2)\\
&\quad\qquad\times \int_0^s \ud r \int_\R \ud w\: G_{\nu/2}(s-r,y-w) 
\frac{1}{\sqrt{s-r}} h(t)\\
&\quad\qquad\times \iint_{\R^2}\mu(\ud z_1)\mu(\ud z_2) \: 
G_{\nu/2}(r,\bar{z}-w)G_{2\nu}(r,\Delta z).
\end{align*}
Integrate first over $\ud w$ using the 
semigroup property of the heat kernel, and then integrate over $\ud r$ using 
\eqref{EH:ErfcGG}, to find that 
\begin{align*}
L_2(t)\le& 
 \pi h(t) \sqrt{t} \int_0^t \ud s \int_\R  \ud y 
\iint_{\R} \ud x_1\ud x_2\; 
G_{\nu/2}(t-s,\bar{x}-y)\: G_{2\nu}(t-s,\Delta x)\phi(x_1)\phi(x_2)\\
&\times \iint_{\R^2}\mu(\ud z_1)\mu(\ud z_2) \: 
G_{\nu/2}(s,\bar{z}-y)G_{2\nu}(s,\Delta z).
\end{align*}
Comparing the above bound with \eqref{EH:L1}, we see that 
\[
L_2(t) \le \pi\sqrt{t} \: h(t) L_1(t)\rightarrow 0, \quad\text{as 
$t\rightarrow 0$.}
\]

{\vspace{1em}\bf\noindent (3)~~} Notice that  $L_3(t)\le \overline{\calH}(t)\: 
L_4(t)$, so 
we only need to consider $L_4(t)$, which is a special case of $L_1(t)$ with 
$\mu(\ud x)=\Vip \ud x$.
Since this $\mu$ belongs to $\calM_H(\R)$, $\lim_{t\rightarrow 0_+} L_4(t)=0$ 
by part (1). This completes the proof of Proposition \ref{PH:Weak}.
\end{proof}

\section{Appendix}

\begin{lemma} \label{LH:MHR}
If $|f(x)|\le c_1 e^{c_2 |x|^a}$ for all $x\in\R$ with 
$c_1,c_2>0$ and $a\in\: ]1,2[\;$, 
then there is $c_3<+\infty$ such that for all $b\in \: ]a, 2[\;$, $|f(x)|\le 
c_3 e^{|x|^b}$ for all $x\in\R$.
\end{lemma}
\begin{proof}
Notice that $c_2|x|^a\ge |x|^b$ if and only if $|x|\le c_2^{\frac{1}{b-a}}$. 
Hence, 
$ c_2|x|^a-|x|^b\le  c_2 \: c_2^{\frac{a}{b-a}}-0= c_2^{\frac{b}{b-a}}$. 
Therefore,
$c_1 \exp\left(c_2 |x|^a - |x|^b\right) \le  c_1 
\exp(c_2^{\frac{b}{b-a}}) =: c_3$.
\end{proof}

\begin{proposition}[Proposition 2.16 of \cite{ChenDalang13Heat}]
\label{PH:G}
There are three universal and optimal constants $C_1=1$, $C_2 
=\frac{\sqrt{2}-1}{\sqrt{\pi}}$, and
$C_3 = \frac{1}{\sqrt{\pi}}$,
such that for all $s,t$ with $0\le s\le t$ and $x\in\R$,
\begin{gather*}
 \int_0^t\ud r\int_\R \ud z \left[G_\nu(t-r,x-z)-G_\nu(t-r,y-z)\right]^2
\le \frac{C_1}{\nu} |x-y|\;,\\
 \int_0^s\ud r\int_\R \ud z \left[G_\nu(t-r,x-z)-G_\nu(s-r,x-z)\right]^2
 \le \frac{C_2}{\sqrt{\nu}} \sqrt{t-s}\;,\\
\int_s^t\ud r\int_\R \ud z \left[G_\nu(t-r,x-z)\right]^2
\le \frac{C_3}{\sqrt{\nu}} \sqrt{t-s}\;,\\
\notag
\iint_{\R_+\times\R} \left(G_\nu(t-r,x-z)-G_\nu(s-r,y-z)\right)^2
\ud
r \ud z
\le 2C_1 \left(\frac{|x-y|}{\nu} +\frac{\sqrt{|t-s|}}{\sqrt{\nu}} \right)\;,
\end{gather*}
where we use the convention that $G_\nu(t,\cdot)\equiv 0$ if $t\le 0$.
\end{proposition}

\begin{lemma}[Lemma 4.3 of \cite{ChenDalang13Heat}]\label{LH:GG}
For all $t$, $s>0$ and $x$, $y\in\R$, we have that $ G_\nu^2(t,x) =
\frac{1}{\sqrt{4\pi\nu t}} G_{\nu/2}(t,x)$ and $G_\nu(t,x)G_\nu\left(s,y\right)
= G_\nu \left(\frac{ts}{t+s},\frac{s x+t
y}{t+s}\right)
G_\nu\left(t+s,x-y\right)$.
\end{lemma}

\begin{lemma}[Lemma 4.4 of 
\cite{ChenDalang13Heat}]\label{LH:Split}
 For all $x$, $z_1$ $z_2\in\R$ and $t,s>0$, denote
$\bar{z} = \frac{z_1+z_2}{2}$, $\Delta z = z_1-z_2$. Then
$G_1\left(t,x-\bar{z}\right)
G_1\left(s,\Delta z\right)
\le \frac{(4t) \vee s}{\sqrt{t s}}
G_1\!\left((4t)\vee s,x-z_1\right)
G_1\!\left((4t)\vee s,x-z_2\right)$, where $a\vee
b:=\max(a,b)$.
\end{lemma}

\begin{lemma}[Lemma 4.9 of \cite{ChenDalang13Heat}]
\label{LH:IntGG}
 For $0\le s\le t$ and $x, \: y\in\R$,  we have that
\[\int_0^t \ud s\:
G_\nu(s,x)G_\sigma(t-s,y) = \frac{1}{2\sqrt{\nu\sigma}}
\Erfc\left(\frac{1}{\sqrt{2t}}\left(\frac{|x|}{\sqrt{\nu}}
+\frac{|y|}{\sqrt{\sigma} }\right)\right),
\]
where $\nu$ and $\sigma$ are strictly positive.
In particular, by letting $x=0$, we have that
\begin{align} \label{EH:ErfcGG}
\int_0^t\ud s\: \frac{G_\sigma(t-s,y)}{\sqrt{2\pi \nu s}}  =
\frac{1}{2\sqrt{\nu\sigma}}\Erfc\left(\frac{|y|}{\sqrt{2\sigma
t}}\right)\le
\frac{\sqrt{\pi t}}{\sqrt{2\nu}} G_{\sigma}\left(t,y\right).
\end{align}
\end{lemma}
Note that the inequality in \eqref{EH:ErfcGG} is because by \cite[(7.7.1), 
p.162]{NIST2010},
\begin{align}
 \label{EH:ErfcBd}
\Erfc(x) = \frac{2}{\pi} e^{-x^2} \int_0^\infty \ud t\: \frac{e^{-x^2 
t^2}}{1+t^2}
\le
\frac{2}{\pi} e^{-x^2} \int_0^\infty\ud t\: \frac{1}{1+t^2} = e^{-x^2}\;,
\end{align}

\begin{lemma}\label{LH:IntGds}
For $t>0$, $\nu>0$ and $x\in\R$,  we have that
\[
\int_0^t  \ud s\: G_\nu(s,x) = 2 t\: G_\nu(t,x) - \frac{|x|}{\nu}\: \Erfc\left( 
\frac{|x|}{\sqrt{2 \nu t}}\right).
\]
\end{lemma}
\begin{proof}
The case where $x=0$ can be easily verified. Assume 
that $x\ne 0$. By change of variables $y = |x|/\sqrt{2\nu s}$ and integration 
by parts, we have that
\[
\int_0^t  \ud s\: G_\nu(s,x) =\int_{\frac{|x|}{\sqrt{2\nu t}}}^{+\infty} \ud 
y\: 
\frac{|x|}{\sqrt{\pi } \:\nu\: y^2} \: e^{-y^2}
= \left.\frac{|x|}{\sqrt{\pi }\:\nu\: y} \: 
e^{-y^2}\right|_{+\infty}^{\frac{|x|}{\sqrt{2\nu t}}}
- \frac{|x|}{\nu} \int_{\frac{|x|}{\sqrt{2\nu t}}}^{+\infty}\ud y\: 
 \frac{2}{\sqrt{\pi}}\: e^{-y^2}\:.
\]
Therefore, the conclusion follows from the definition of the function 
$\Erfc(\cdot)$.
\end{proof}

\begin{lemma}\label{L2:gtxr}
If $\nu>0$, $t>0$, $n>1$ and $x\in\R$, then for $r\in \left[0,n^2 t\right]$,
\[
\left|\frac{G_{\nu/2}\left(t+r,
x\right)}{G_{\nu/2}\left(t,x\right)}
-1\right|
 \le
\frac{3r}{t+r}\exp\left(\frac{n^2 x^2}{\nu
t\left(1+n^2\right)} \right)
 \le
\frac{3}{2}\frac{\sqrt{r (1+n^2)}}{\sqrt{t}} G_{\frac{\nu}{2}}^{-1}(t,x)
\:G_{\frac{\nu(1+n^2)}{2}}(t,x).
\]
\end{lemma}
\begin{proof}
Fix $\nu>0$, $t>0$, $n>1$, and $x\in\R$.
For $r\in \left[0,n^2 t\right]$, define
\[
g_{t,x}(r) =
\frac{G_{\nu/2}\left(t+r,
x\right)}{G_{\nu/2}\left(t,x\right)} -1
=
\frac{\sqrt{t}}{\sqrt{t+r}}\exp\left(\frac{x^2}{\nu t}
\frac{r}{t+r}\right) - 1\:.
\]
Clearly $g_{t,x}(0)=0$.
Notice that
\[
\left|g_{t,x}(r)\right|\le
\left|\exp\left(\frac{x^2}{\nu t} \:
\frac{r}{t+r}\right)-1\right| +
\exp\left(\frac{x^2}{\nu t} \: \frac{r}{t+r}\right)
\left|\frac{\sqrt{t}}{\sqrt{t+r}}-1\right|.
\]
The second term on the right-hand side is  bounded by
$\exp\left(\frac{n^2 x^2}{\nu (1+n^2) t}\right) \frac{r}{t+r}$ for all  $r\in
\left[0,n^2 t\right]$,
because $\frac{r}{r+t} \in \left[0,
\frac{n^2}{1+n^2}\right]$ for $r\in \left[0,n^2 t\right]$.
To bound the first term, we use the fact that
for fixed $a>0$ and $b>0$, $0\le e^{a h} -1 \le e^{a b} \frac{h}{b}$
for all $h\in [0,b]$.
Apply this fact to $\exp\left(\frac{x^2}{\nu t} \:
\frac{r}{t+r}\right)-1$ with $a=\frac{x^2}{\nu t}$, $h=\frac{r}{r+t}$
and $b=\frac{n^2}{1+n^2}$: the first term is bounded by
$2 \exp\left(\frac{n^2 x^2}{\nu t\left(1+n^2\right)}\right) \frac{r}{r+t}$
for all $r\in \left[0,n^2 t\right]$.
Adding these two bounds proves the first inequality.
The second one follows from the inequality $t+r\ge 2\sqrt{t r}$.
\end{proof}

\begin{lemma}\label{L2:arcsin}
 $\int_{t}^{t'} \frac{1}{\sqrt{s(t'-s)}} \ud s = 2
\arcsin\left(\sqrt{\frac{t'-t}{t'}}\right)$ for all $t'>0$ with $t'\ge t\ge 0$.
\end{lemma}
\begin{proof}
For $t=0$, the l.h.s. reduces to the Beta integral
(see, e.g., \eqref{E2:BetaInt}. If $t\in \; 
]0,t']$,
differentiate with respect to $t$ on
both sides.
\end{proof}

\addcontentsline{toc}{section}{Bibliography}
\def\polhk#1{\setbox0=\hbox{#1}{\ooalign{\hidewidth
  \lower1.5ex\hbox{`}\hidewidth\crcr\unhbox0}}} \def\cprime{$'$}

\end{document}